\documentclass{amsart}

\usepackage{hyperref}

\usepackage{amsmath, mathrsfs, amssymb,amsthm,hyperref,centernot,comment,stmaryrd,enumerate,cancel,mathtools,bbold,pbox,multirow}

\hypersetup{pdfstartview={XYZ null null 1.25}}
\usepackage[all]{xy}
\title{Seurat games on Stockmeyer graphs}

\theoremstyle{plain}
\newtheorem{thm}{Theorem}[section]
\newtheorem{prop}[thm]{Proposition}
\newtheorem{cor}[thm]{Corollary}
\newtheorem{lem}[thm]{Lemma}
\newtheorem{ex}[thm]{Example}
\theoremstyle{definition}
\newtheorem{defn}[thm]{Definition}
\newtheorem{conj}[thm]{Conjecture}

\newcommand{\ind}{\mathsf{in}}
\newcommand{\outd}{\mathsf{out}}

\newcommand{\odd}{\mathsf{odd}}
\newcommand{\G}{{\bf{G}}}
\newcommand{\M}{{\bf{MSO}}}
 \def\set#1{{\{ #1\}}}
 \def\tilg{{\widetilde{\Gamma(G)}}}

\def\ws{{winning strategy}}

\author{Rob Egrot}
\address{Faculty of Information And Communication Technology, Mahidol University}
\email{robert.egr@mahidol.ac.th}

\author{Robin Hirsch}
\address{Department of Computer Science University College London}
\email[corresponding author]{r.hirsch@ucl.ac.uk}

\begin{document}

\maketitle

\begin{abstract}
We define a family of vertex colouring games played over a pair of graphs or digraphs $(G,H)$ by players $\forall$ and $\exists$. These games arise from work on a longstanding open problem in algebraic logic. It is conjectured that there is a natural number $n$ such that $\forall$ always has a winning strategy in the game with $n$ colours whenever $G\not\cong H$. This is related to the reconstruction conjecture for graphs and the degree-associated reconstruction conjecture for digraphs. We show that the reconstruction conjecture implies our game conjecture with $n=3$ for graphs, and the same is true for the degree-associated reconstruction conjecture and our conjecture for digraphs. We   show (for any $k<\omega$)  that the 2-colour game can distinguish certain  non-isomorphic pairs of graphs that cannot be distinguished by the $k$-dimensional Weisfeiler-Leman algorithm.  We also show that the 2-colour game can distinguish the non-isomorphic pairs of graphs in the families defined by Stockmeyer as counterexamples to the original digraph reconstruction conjecture.         
\end{abstract}

\section{Introduction}

An Ehrenfeucht-Fra\"iss\'e game is played over a pair of structures by two players, $\forall$ and $\exists$,  who place matching pebbles on the two structures, and test their equivalence with respect to a first-order language whose variables correspond to the pairs of pebbles.    In this paper, we investigate a similar game, where the two players have $k$ colours which are used (and may be reused) to paint \emph{sets} of points from the two structures, rather than placing pebbles on individual points.  Section \ref{S:SG} defines the game  precisely, but we will sketch out the main idea here.   We restrict our attention to graphs, which may or not be  directed.  In the game ${\bf G}^k(G, H)$,  the first player $\forall$ is trying to prove that  graphs $G, H$ are not isomorphic, and $\exists$ is trying to prevent this. In pursuit of these competing goals, the players take it in turns to `paint' sets of vertices of a graph  using one of the $k$ colours, with the second player trying to match the move of the first as best she is able, in the other graph.  
  $\forall$ wins the game  if $\exists$ fails to match moves, either because there is a node painted by a certain combination of colours in one graph but no such node in the other, or because there is an edge in one graph, but no edge in the other graph matching the colour combinations of source and target nodes, see Section \ref{S:SG} for the full definition of the game.    Versions of this colouring game, generalised from graphs to binary structures, and also specialised to sets, are used  in \cite{EgrHirNonFin, EgrHirNote} to prove some results in algebraic logic (see Section \ref{S:ax} for a brief discussion of this).

At the moment, not much is  known about this colouring game (which we refer to as a \textbf{Seurat game}, in reference to pointillist painting), though there are some similarities with  Ehrenfeucht-Fra\"iss\'e style games for monadic second-order logic, which we discuss later.  If $\forall$ can  force a win, then the two graphs cannot be isomorphic, as with isomorphic graphs $\exists$ may always perfectly mirror $\forall$'s moves. However, it is not clear to what extent the converse holds. In other words, whether graphs exist that are not isomorphic but where nevertheless $\exists$ can play indefinitely without losing. Intuitively,  $\forall$ should be able to win more easily with more colours available, as he can force more complicated situations which must be mirrored between graphs. For example, with only a single colour his ability to win is rather limited, and it is easy to construct examples of non-isomorphic graphs where he does not have a winning strategy (the graphs may even have different cardinalities). However, it is currently not known whether there is some $n$ such that $\exists$ being able to avoid losing in  ${\bf G}^n(G, H)$ always implies $G\cong H$. In particular, we are not aware of any pair of non-isomorphic graphs where $\exists$ can avoid losing in even the game with only two colours. We can reformulate this as a conjecture as follows.
\begin{conj}\label{conj}
There is $n\geq 2$ such that for all digraphs $G,H$, if $G\not\cong H$ then $\forall$ has a winning strategy in  ${\bf G}^n(G,H)$.
\end{conj} 

This conjecture has a certain thematic similarity to the famous reconstruction conjecture \cite[p. 29]{Ulam60}, about which much has been written (see e.g. \cite{O'Ne70,Har74,Stock88,Bon91,LauSca16} for exposition). This was pointed out to the second listed author by A. Dawar. This connection is most explicit when the reconstruction conjecture is phrased in the following way.

\begin{defn}\label{D:reconstruction}
The \textbf{reconstruction conjecture} is that if $G$ and $H$ are non-isomorphic (undirected) graphs with at least one having at least three vertices, then there is a graph $F$ such that the number of point-deleted subgraphs of $G$ that are isomorphic to $F$ is not equal to the number of point-deleted subgraphs of $H$ that are isomorphic to $F$.
\end{defn}

The multiset of point-deleted subgraphs (up to isomorphism) of a graph is known as its \textbf{deck}, and so the reconstruction conjecture phrased this way says that we can detect if two graphs with at least three vertices are not isomorphic by comparing their decks, just as our conjecture above says that if two digraphs are not isomorphic we can detect this in a certain game. There is a version of the reconstruction conjecture for digraphs obtained by replacing `graph' with `digraph' everywhere in Definition \ref{D:reconstruction}. The reconstruction conjecture for digraphs is known to be false, with families of counterexamples provided in \cite{Stock77,Stock81}. The original graph reconstruction conjecture, however, remains open.

While the reconstruction conjecture for digraphs is false, a related conjecture does remain open. This is often referred to as the \textbf{degree-associated reconstruction conjecture for digraphs} (see Definition~\ref{D:deg_rec} below), and is due to Ramachandran \cite{Rama81,Rama83}. This conjecture is stronger than the reconstruction conjecture for graphs (if we view graphs as special kinds of digraphs), but weaker than the reconstruction conjecture for digraphs (which, as mentioned previously, is false).   We show in Section \ref{S:rec} that if the degree-associated reconstruction conjecture for digraphs is true then so is Conjecture \ref{conj}, and similarly if the original reconstruction conjecture is true then  the version of Conjecture \ref{conj} for undirected graphs holds (both with $n=3$). It follows immediately that a counterexample to Conjecture \ref{conj} or its graph analogue would disprove the degree-associated reconstruction conjecture or the reconstruction conjecture, respectively.

\medskip

The bulk of this paper is devoted to investigating  Seurat games.   Section \ref{S:ax} describes the origin of these Seurat games in a problem to do with axiomatisations of the class of representable relation algebra.
The basic definitions of Seurat games are provided in Section \ref{S:SG}, along with several general results about strategy, and comparisons with certain games for monadic second-order logic. To demonstrate the power of these games to distinguish non-isomorphic graphs, and provide support for Conjecture \ref{conj}, we provide two main examples. 

First, in Section \ref{WL} we study the graph constructions from \cite[Section 6]{CFI92}. These constructions were introduced to demonstrate that the $k$-dimensional Weisfeiler-Leman algorithm ($k$-WL for short - see e.g. \cite{Kief20} for a summary) is not sufficient to distinguish all non-isomorphic graphs for any $k$. Our result here is to use this construction to produce, for each $k$, a pair of graphs that cannot be distinguished by $k$-WL, but can be distinguished in the 2-colour Seurat game. 
     
We also consider certain pairs of non-isomorphic digraphs from  \cite{Stock81}, which we call  \textbf{Stockmeyer graphs}.  Stockmeyer proved that these pairs are not isomorphic to each other but share the same deck.  We describe these Stockmeyer graphs in Section~\ref{S:stock}.  Although these pairs share the same deck, we prove in Section \ref{S:games} that $\forall$ has a winning strategy in the 2-colour Seurat game over them.        Proving this is not entirely straightforward, and to this end we define in Section \ref{S:tally} something called a \textbf{tally-sequence}. This extends the notion of the degree sequence of a vertex, and is related to the \textbf{colour refinement} algorithm (which is the $k=1$ case of $k$-WL). The benefit of computing tally-sequences over the more well known colour refinement algorithm is that we can prove that $\exists$ must, in a technical sense to be defined later, match tally sequences in her moves if she doesn't want to lose. It is by exploiting this constraint that we are able to prove that $\forall$ has a winning strategy in the 2-colour game over any Stockmeyer pair. It is not currently known whether $\exists$ must also match colour refinement colours. This is discussed in more detail in Section \ref{S:tally}, particularly just after Lemma \ref{L:2-colours}.  
As mentioned previously, further support for Conjecture \ref{conj} is provided in Section \ref{S:rec}, which contains a discussion of the relationship between the reconstruction conjectures and our conjecture about Seurat games.

\section{Axiomatising  the class of Representable Relation Algebras}\label{S:ax}

Seurat games arise from a  problem in Relation Algebra, which we outline in this section.  The material here is not required in order to read the rest of the paper.   A historical overview of Relation Algebra, whose origins go back to the work of De Morgan in the mid 19th century, can be found in \cite{Mad91,S-A05}, and also in the introduction to \cite{HirHod02}. We will content ourselves here with only a brief synopsis, beginning our story in the early 1940s  with Tarski's proposal \cite{Tar41} for a specialized calculus of binary relations. This calculus took the form of a kind of pared-down formal logic, with axioms extending those for propositional logic, and two rules of inference. The formal treatment of binary relations can of course be handled in a fragment of first-order logic, with which Tarski was well familiar. However, Tarski had a particular interest in an elegant formalization specific to binary relations, as set theory is developed using a single binary relation, so by studying the logic of binary relations he was, in a sense, studying set theory, and thus all of mathematics.

Over the course of the decade, Tarski and his circle investigated several algebraic versions of his calculus, before settling on a form that  appeared in print in \cite{JT48, ChinTar51}.
 In this abstract and paper, the modern definition of a relation algebra is set down. This modern formulation defines relation algebras as a variety (i.e. equationally defined class) of algebraic structures. The details are not important here, but a relation algebra in Tarski's sense is a Boolean algebra with additional operators in which a finite number of additional equations hold. This turns out to be capable of expressing all properties of binary relations involving at most three variables. 

Since Tarski intended his calculus of relations to capture the logic of binary relations, a natural question to ask is whether relation algebras as described here successfully capture all and only the true properties of binary relations, or at least those true properties of binary relations expressible with three variables. Unfortunately, the answer to this question was fairly quickly discovered to be ``no" \cite{Lyn50}. We omit the details, but the root of the problem is that not every relation algebra is isomorphic to a `concrete' algebra of binary relations over a set. The relation algebras which \emph{are} so isomorphic are therefore distinguished from relation algebras in general by being termed \textbf{representable}. The class of representable relation algebras ($RRA$) is thus a strict subclass of the class of all relation algebras ($RA$), and, rephrasing, the problem is that not every relation algebra is representable. 

What \emph{is} captured by $RA$ can be summarized by the following two facts: 
\begin{enumerate}[(1)]
\item\label{en:1} Every equation in the language of relation algebras can be translated into a first-order statement about binary relations involving at most three variables, and conversely, every first-order statement about binary relations involving at most three variables can be translated into an equation in the language of relation algebras. 
\item \label{en:2} An equation is true in all relation algebras if and only if its translation into a first-order statement about binary relations is \textbf{provable} using at most four variables. 
\end{enumerate}

According to historical remarks in \cite[pp. 88–9] {TG87} and \cite[pp. 28–9, 529]{Mad06} \/ \eqref{en:1} was proved sometime in 1942--1943 and first published in \cite[Theorems 3.9(viii)(ix)]{TG87}. \/ \eqref{en:2} was attributed to \cite{Mad78} by  \cite[pp. 92–3, 209]{TG87} and first published in \cite[Theorem 24]{Mad:89}.
Tarski's choice of axioms for $RA$ turns out to be essentially optimal for a finite theory. The argument is too technical to go into here, but the result is that to define relation algebras so that the statement obtained from (2) above by replacing `four' with `five' holds would require an infinite number of additional axioms.  In other words, while validity over $RA$ does not capture all true properties of binary relations that can be stated with three variables, it does capture, for classical proof systems, all true three variable  properties provable with four variables, and no finite set of axioms can capture the true three variable properties classically provable with five variables. It's hard to find a precise statement of this fact in the literature,  but the argument can be reconstructed by reading \cite[Section 6]{HirHod01II} and following up some of the references to be found there.

From the point of view of `true properties of binary relations' the class of main interest is $RRA$, so it is natural to look for axioms for this class. It turns out that $RRA$ is also a variety \cite{Tar55}, and is even recursively axiomatizable by equations \cite[Theorem 8.4]{HirHod02}, though axiomatizing it requires an infinite number of equations, as mentioned above. Aside from this good news about the existence of an equational axiomatization, most news about the axiomatizability of $RRA$ turns out to bad. For example, $RRA$ is not even finitely axiomatizable in first-order logic \cite{Mon64}, and it can't be axiomatized by equations using only a finite number of variables \cite[Theorem 3.5.6]{Jon91}, nor by any set of equations only a finite number of which are non-canonical \cite{HodVen05}. An outstanding question is whether it can be axiomatized in first-order logic using only a finite number of variables. This is mentioned as Problem 17.4 in \cite{HirHod02}, which also provides a strategy for a potential proof that no such axiomatization exists (see page 625 of that book, and also the brief discussion in \cite[p491]{S-A05}). Unfortunately, the strategy as described there doesn't quite work, in the sense that solving the problem defined in the book does not lead to a proof that $RRA$ has no finite variable first-order axiomatization. Corrections to this problem are discussed in \cite{EgrHirNote}, where  the following is proved.

\begin{thm}\label{T:ax}
Suppose there exist two finite digraphs graphs $G$ and $H$ such that,
\begin{enumerate}[1.] 
\item \label{en:strat} $\exists$ has a winning strategy in  ${\bf G}^k(G,H)$ ,
\item every partial homomorphism $\{(i,j),(i',j')\}$ where $i\neq i'\in H$ of $H$ to itself extends to a full homomorphism on $H$, and 
\item \label{en:no h}  there are $i\neq i'\in G$ and $j,j'\in H$ such that $\{(i,j),(i',j')\}$ is a partial homomorphism that does not extend to a homomorphism $G\to H$.
\end{enumerate}  
Then $RRA$ has no  $(k-3)$-variable first-order axiomatization.
\end{thm}
A condition that suffices to establish part  \eqref{en:no h} in the cases of interest is that there is no homomorphism from $G$ to $H$. This condition does not imply \eqref{en:no h} for all graphs, because  if $G$ is complete and $H$ is edgeless then there is no homomorphism, yet \eqref{en:no h} fails since there are no partial homomorphisms of size two.   However, if we assume \eqref{en:strat}, this exceptional case is excluded, and so the absence of a homomorphism does imply \eqref{en:no h}.
  
 \cite{EgrHirNonFin} also uses Seurat games to prove another negative result about axiomatizations of $RRA$, namely that any first-order axiomatization requires sentences of arbitrary quantifier depth.

Note that from the point of view of Theorem \ref{T:ax} we hope Conjecture \eqref{conj} is false, as it is $\exists$ we wish to have a winning strategy in games over pairs of non-isomorphic graphs.   However, it follows from our results here that the falsity of \eqref{conj} would disprove the degree-associated reconstruction conjecture, so we do not expect this to be simple.  
But, more optimistically, thinking about the colouring game could potentially provide insight leading to a disproof of one or both of the reconstruction conjectures, if indeed one or both are false.

\section{Seurat games}\label{S:SG}

A \textbf{Seurat game} $\G^k(G, H)$  is played by two players, $\forall$ and $\exists$, using a set ${\bf Col}$ of $k$ colours (where $k\geq 1$), over a pair of graphs $(G,H)$. We assume that all graphs are finite, and we allow self-edges (loops), but disallow multiple edges.    A position $(g, h)$ in the game consists of a pair of functions $g:{\bf Col}\rightarrow \wp(G),\; h:{\bf Col}  \rightarrow\wp(H)$ (here and elsewhere we identify a graph with its set of nodes).
There are $\omega$ rounds in a play of the game.  In each round, if the current position is $(g, h)$, \/  $\forall$ chooses a colour $c\in{\bf Col}$,  either graph $G$ or $H$, and a subset of the nodes of his chosen graph, \/ player $\exists$ chooses a subset of the nodes of the other graph.    The new position at the end of the round is $(g',h')$ where $g', h'$ are identical to $g, h$ (respectively) on colours other than $c$, and $g'(c) \subseteq G,\; h'(c)\subseteq H$ are given by the two chosen sets.   Note that reusing a colour erases its first use, so, for example, if $\forall$ colours set $S$ of vertices of $G$ red, and then later colours a set $T$ of vertices of $G$ red, the vertices in $S$ are no longer red, unless they are also in $T$. At the end of each round, each vertex of $G$ and $H$ will be coloured with between 0 and $k$ colours. A \textbf{palette} is a subset of ${\bf Col}$. When we talk about the palette of a vertex we mean the set of colours applied to it (which could be empty). Given a palette $P$ of colours we define the \textbf{range} of $P$ in $G$, which we denote $P^G$, to be the set of vertices of $G$ whose palette is $P$, and we define the palette of $P$ in $H$, denoted $P^H$, similarly. The game opens with round 0. We say that $\forall$ wins in round $n$ if at the beginning of that round, any of the following conditions is satisfied:
\begin{enumerate}[(C1)]
\item There is a palette $P$ such that $P^G$ is empty and $P^H$ is not, or vice versa.
\item There are palettes $P_1$ and $P_2$ such that there is an edge from $P^G_1$ to $P^G_2$ but no edge from $P^H_1$ to $P^H_2$, or vice versa.
\end{enumerate}
Initially, all colours colour the empty set of nodes, so it is impossible for $\exists$ to lose in round 0, though if we used a variation of the game allowing other starting configurations then this would not necessarily be the case.  Observe that the rules of the Seurat game apply equally well to directed and undirected graphs. Now, $\exists$ cannot win outright, but we say she has a \textbf{winning strategy} if she can play in such a way that she never loses (i.e. so that (C1) and (C2) are never true). Alternatively, $\forall$ has a winning strategy if he can guarantee that he will win in a finite number of rounds. Note that, as a consequence of K\"onig's Tree Lemma \cite{Kon26}, exactly one player has a winning strategy in each game. 

Clearly, if $G\cong H$, then $\exists$ has a winning strategy in $\G^k(G, H)$ for all values of $k$, as she can just copy $\forall$'s moves. It is not known for what values of $k$, if any, the converse is true. In other words, is there a value of $k$ such that whenever $\exists$ has a winning strategy in $\G^k(G,H)$ it is necessarily the case that $G\cong H$? If so, then clearly for all $k'\geq k$ if $\exists$ has a \ws\ in $\G^{k'}(G, H)$ then she also has a \ws\ in $\G^k(G, H)$ as we remarked before, hence $G\cong H$.  Thus, if the converse implication holds for some $k$ then it also holds for all $k'\geq k$. On the other hand, there are no known examples of $G$ and $H$ with $G\not\cong H$ but where $\exists$ has a winning strategy in $\G^k(G,H)$ for even $k= 2$. 

The case $k=1$ with a single colour is rather straightforward, however.  For the one colour game, observe in each round that the colour is deleted from both graphs before being reassigned, so the game effectively restarts from the beginning. Hence if $\exists$ has a strategy to ensure surviving just one round then she has a \ws\ for the $\omega$ length game.    Suppose $G$ and $H$ are edgeless graphs with   two or more vertices. In any round the set of nodes chosen by $\forall$  (in either graph) has either no nodes (empty), all nodes, or some but not all nodes of the graph included.  Provided $\exists$ chooses a set of nodes of the other graph correspondingly including none, all, or some but not all nodes, she will survive.  So there are plenty of pairs of non-isomorphic graphs that cannot be distinguished by the one colour game.

Some of the work in \cite{EgrHirNote} implicitly adds some winning conditions for $\forall$ to the Seurat game for graphs, though this is not immediately apparent as in that paper Seurat games are defined for binary structures (i.e. relational structures with one or more binary relations). We call the result of adding these conditions the \textbf{strong Seurat game}, and denote by $\hat{\G}^k(G,H)$ (for $k$ colours). This is exactly like the Seurat game just described, but with the following additional winning conditions for $\forall$. 
\begin{enumerate}
\item[(C3)] There are palettes $P_1$ and $P_2$ such that every vertex in $P_2^G$ is the target of an edge out of $P_1^G$, but there is a vertex in $P_2^H$ that is not the target of any edge out of $P_1^H$, or vice versa switching $G$ and $H$.
\item[(C4)] There are palettes $P_1$ and $P_2$ such that every vertex in $P_1^G$ is the origin of an edge into $P_2^G$, but there is a vertex in $P_1^H$ that is not the origin of any edge into $P_2^H$, or vice versa switching $G$ and $H$.
\end{enumerate}

For the work in \cite{EgrHirNote} with binary structures, conditions (C3) and (C4) are not needed explicitly as a graph can be augmented with an additional binary relation corresponding to `missing' edges. In these augmented graphs, if (C3) or (C4) is triggered for edges then (C2) is triggered for `missing' edges, so the extra conditions add nothing. 

For ordinary graphs, the strong Seurat $\hat{\G}^k(G,H)$ game is obviously at least as powerful as the Seurat game $\G^k(G,H)$ for distinguishing $G$ and $H$, but it is not clear whether it is strictly more powerful for $k\geq 2$, as we do not currently know of any non-isomorphic graphs that cannot be distinguished by $\G^2(G,H)$. When $k=1$, Example \ref{E:k1} below demonstrates that the strong game is strictly more powerful. We also have the following result indicating that $\G^k$ is, at worst, not far behind $\hat{\G}^k$.

\begin{lem}
If $\forall$ has a strategy for winning by round $n$ in $\hat{\G}^k(G,H)$, then $\forall$ also has a strategy for winning by round $n+1$ in $\G^{k+1}(G,H)$.
\end{lem} 
\begin{proof}
$\forall$'s strategy is to play as in his strategy for $\hat{\G}^k(G,H)$ using the first $k$ colours. If this strategy wins by either (C1) or (C2) then he has nothing more to do. Alternatively, if his strategy wins by (C3) then there are palettes $P_1$ and $P_2$ and, without loss of generality, a vertex $v$ of $P_2^H$ that is not the target of any edge coming out of $P_1^H$ in $H$, while every vertex in in $P_2^G$ is the target of an edge out of $P_1^G$. In his next move, $\forall$ colours $v$ with the as yet unused $(k+1)$th colour. Then $\exists$ must respond by colouring a vertex of $P_2^G$, or else she violates (C1), but then she loses by (C2) anyway. The (C4) case is similar. 
\end{proof}

\begin{ex}\label{E:k1}
Here we present an example demonstrating that the strong game with one colour $\hat{\G}^1$ is strictly stronger than the standard game $\G^1$. Note first that when there is only one colour, after that colour has coloured some non-empty subset of a graph, there are exactly four possible `edge types' in the graph. If we let $c$ stand for \textbf{coloured}, and $u$ stand for \textbf{uncoloured}, these edge types are $(c,c), (c,u), (u,c), (u,u)$. This produces a total of $2^4=16$ possible `edge type' combinations in a given graph, though not every graph can witness every combination.  

Let $G$ and $H$ be as in Figure \ref{F:GH}. So $G$ is the disjoint union of two 3-vertex chains, and $H$ is similar but with the addition of a disjoint 2-vertex cycle. A tedious check reveals that  a non-empty set of nodes in $G$ or $H$ can witness precisely the same 12 out of the possible 16 edge type combinations (the omitted combinations being $\{(u,u)\}$, $\{(c,u)\}$, $\{(u,c)\}$ and $\emptyset$). Thus $\exists$ has a winning strategy in $\G^1(G,H)$, as whatever move $\forall$ makes she just plays a move witnessing the same edge type combination.  

However, in a play of  $\hat{\G}^1(G, H)$,  if $\forall$ colours the 2-vertex cycle in $H$, then $\exists$ must respond by colouring a non-empty set of vertices in $G$ such that every coloured vertex is the target of an edge from another coloured vertex, else she will trigger (C3). But this is impossible, so $\forall$ has a winning strategy in $\hat{\G}^1(G,H)$.
\end{ex}

\begin{figure}[h]
\[\xymatrix{\bullet\ar[r] & \bullet\ar[r]  & \bullet & & \bullet\ar[r]  & \bullet\ar[r]  & \bullet \\
\bullet\ar[r]  & \bullet\ar[r]  & \bullet & & \bullet\ar[r]  & \bullet\ar[r]  & \bullet \\
& & & & & \bullet\ar@/^/[r]  & \bullet\ar@/^/[l]\\
& G & & & &  H
}\]
\caption{The graphs $G$ and $H$}
\label{F:GH} 
\end{figure}
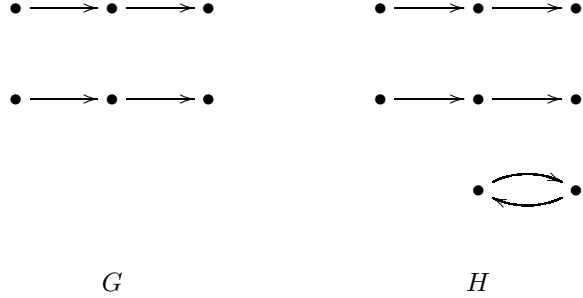  

Similar games have been studied in the context of monadic second-order logic. Let $m,k< \omega$, and let $G,H$ be graphs. Then $\M^k_m(G,H)$ is a game played over $G$ and $H$ between players $\forall$ and $\exists$ using $k$ colours and $m$ pairs of pebbles. The rules of are similar to those of $\G^k$. Each round begins with $\forall$ choosing either a colour or a pebble pair. He then either colours a set of vertices of either $G$ or $H$ with his chosen colour, or places one of his chosen pebbles on a single vertex of either $G$ or $H$. In response, $\exists$ must either colour a subset of the other graph with the same colour, as in the Seurat game, or place the other pebble from the pair on a single vertex. Colours are interpreted as instantiations of monadic predicates. In this game, $\forall$ wins in round $n$ if, at the start of that round, the partial map induced by matching pairs of pebbles is not an isomorphism (taking into account the predicates induced by the colours, as well as the edge relation). It is known that $\M^k_m$ captures expressibility in monadic-second logic for relational structures with $k$ second-order and $m$ first-order variables (see e.g. \cite[Section 7]{Lib04}). The following proposition makes explicit a connection between the monadic second-order games and the Seurat games.

\begin{prop}\label{P:MSO}
Let $k, m< \omega$, and let $G,H$ be graphs. Then:
\begin{enumerate}
\item If $\forall$ has a winning strategy for the strong Seurat game $\hat{\G}^k(G,H)$, then he also has a winning strategy for $\M^k_2(G,H)$.
\item If $\forall$ has a winning strategy for $\M^k_m(G,H)$ and $k+m \geq 2$, then he also has a winning strategy for the ordinary Seurat game $\G^{k+m}(G,H)$.
\end{enumerate}
\end{prop}
\begin{proof}
Suppose first that $\forall$ has a winning strategy for $\hat{\G}^k(G,H)$. His strategy for $\M^k_2(G,H)$ is as follows. Initially, every round he plays with colours exactly as he would in $\hat{\G}^k(G,H)$. Since he is playing according to a winning strategy, at some point one of (C1)-(C4) will be triggered. If (C1) is triggered, then, without loss of generality, there's a palette $P$ such that $P^G$ is empty and $P^H$ is not. $\forall$ then plays a pebble move using one of the vertices in $P^H$. Since $P^G$ is empty, wherever $\exists$ places the partner pebble, the induced map cannot be an isomorphism, and so she loses. 

Alternatively, if (C2) is triggered then, without loss of generality, there are palettes $P_1$ and $P_2$ such that there's an edge from $P^G_1$ to $P^G_2$ but no edge from $P^H_1$ to $P^H_2$. Here $\forall$ chooses $u\in P^G_1$ and $v\in P^G_2$ where $(u, v)$ is an edge.   First he plays a pebble move using $u$, and $\exists$ must respond by placing the partner pebble on some $u'\in P^H_1$. Then he plays a pebble move using the remaining pebble pair and the vertex $v$, and $\exists$ must respond by placing the partner pebble on some $v'\in P^H_2$. But now she loses, as the map induced by the pebbles cannot be an isomorphism, as there's an edge $(u,v)$ but no edge $(u',v')$. The arguments for (C3) and (C4) are similar to the one for (C2), so we omit them for brevity. This proves (1).

For (2), note that it follows trivially from $k+m\geq 2$ and Proposition \ref{P:constraints} (S1) below that if $\forall$ colours a single vertex in  $\G^{k+m}(G,H)$, then $\exists$ must respond by colouring a single vertex too, otherwise $\forall$ can force a win.  So, if $\forall$ uses the first $k$ colours of $\G^{k+m}(G,H)$ (call these $c_1,\ldots,c_k$) to correspond to the $k$ colours of $\M^k_m(G,H)$, and the $m$ additional colours (call these $c_{k+1}, \ldots ,c_{k+m}$) to correspond to the pebble moves (i.e. by colouring single vertices), a sequence of moves in $\M^k_m(G,H)$ corresponds directly to a sequence of moves in  $\G^{k+m}(G,H)$. So $\forall$ can just use his strategy for $\M^k_m(G,H)$ without modification (unless $\exists$ breaks the correspondence by colouring more than one vertex in response to a `pebble' move, in which case $\forall$ forces a win as discussed above).   We need only check that a win for $\forall$ in $\M^k_m(G,H)$ translates into a win in $\G^{k+m}(G,H)$.   Suppose then that $\forall$ wins in $\M^k_m(G,H)$ due to the placement of pebble pairs $1,\ldots, m$   on $u_1, \ldots, u_m$ in $G$, and on $u'_1,\ldots, u'_m$ in $H$, so that $u_i\mapsto u'_i\; (1\leq i\leq m)$ is not an isomorphism with respect to colours $\set{c_1, \ldots, c_k}$ and the edge relation. 

There are two cases. Suppose first, for some  $1\leq i\leq m$, that $u_i$   is coloured by a different combination of colours from $\set{c_1, \ldots, c_k}$ than $u'_i$. Then since $u_i, u'_i$ are the unique points coloured with $c_{k+i}$, (C1) must be triggered, so $\forall$ wins $\G^{k+m}(G,H)$. Alternatively,  suppose that the $\set{c_1, \ldots, c_k}$-colour combinations of $u_i$ and $u'_i$ match for all $1\leq i\leq m$, but the induced map is not an isomorphism due to there being an edge $(u_i,u_j)$ in $G$, but no edge $(u'_i,u'_j)$ in $H$  or vice versa (for some $1\leq i, j\leq t$).   Then, $c_{k+i}$ only colours $u_i$ and $u'_i$, and $c_{k+j}$ only colours $u_j$ and $u'_j$,  so (C2) must be triggered, and $\forall$ also wins $\G^{k+m}(G,H)$ in this case too.      
\end{proof}

The 1-colour Seurat game is obviously fairly limited in its ability to distinguish non-isomorphic graphs, as we saw. However, even the 1-colour game can detect some differences, as we see in the next proposition.  Recall that a directed graph is  weakly connected if for every pair of vertices $u, v$ there is a path, of the symmetric closure of the edge relation, from $u$ to $v$, and it is 
strongly connected if it contains a directed path  in the edge relation from $u$ to $v$ and a directed path from $v$ to $u$, for every pair of vertices $u, v$.
\begin{prop}\label{P:1-colour}
Let $G$ and $H$ be digraphs. Then $\forall$ has a winning strategy in $\G^1(G,H)$ in all the following situations:
\begin{enumerate}
\item When $G$ has exactly one vertex and $H$ has more than one (or vice versa). \label{OneV}
\item When one of $G$ or $H$ is strongly connected but the other is not. \label{OneSC}
\item When one of $G$ or $H$ is weakly connected but the other is not. \label{OneWC}
\item When one of $G$ or $H$ has an irreflexive vertex but the other does not. \label{OneI}
\item When $G$ and $H$ are not isomorphic and each has at most two vertices. \label{One2}
\end{enumerate}
\end{prop}
\begin{proof}\mbox{}
\begin{enumerate}[(1)]
\item Here $\forall$ just colours one vertex of $H$. Then $H$ has both coloured and uncoloured vertices, which $\exists$ cannot replicate in $G$.
\item Suppose $G$ is strongly connected but $H$ is not. Let $E_H^\ast$ be the reflexive transitive closure of the edge relation of $H$, and for all $v\in H$ define $E_H^\ast(v) =\{u\in H: (v,u)\in E_H^\ast\}$. Then, as $H$ is not strongly connected, there is $v_0\in H$ such that $E_H^\ast(v_0)\neq H$. So, in his first move $\forall$ colours $E_H^\ast(v_0)$. Now, $\exists$ must respond by colouring a subset $S$ of $G$. If $S$ is not a proper subset of $G$ then $\exists$ loses as there will be an uncoloured vertex of $H$ but not of $G$. But if $S$ is a proper subset of $G$ she will also lose, as then there will be a coloured-to-uncoloured edge in $G$, but not in $H$.
\item Suppose $G$ is weakly connected but $H$ is not. The argument is similar to that used for (\ref{OneSC}), but this time we start with the relation $\hat{E}_H$ on the vertices of $H$ by $\hat{E}_H(u,v)$ if and only if either $E_H(u,v)$ or $E_H(v,u)$. 
\item Here $\forall$ just colours an irreflexive vertex in the  graph which has an irreflexive node.  If $\exists$ plays the empty set she loses by (C1), if she plays a non-empty subset of the  (reflexive) nodes of the other graph she loses by (C2).

\item  Suppose $G, H$ are not isomorphic and both have one or two vertices. Cases where the cardinalities are different are covered by part \eqref{OneV}, the case where both graphs have a single node (reflexive in one case irreflexive in the other) is covered by \eqref{OneI}, so suppose each graph has two nodes.  $\forall$ colours an arbitrary single node $\set{g}$ where $g\in G$, let the other node of $G$ be $g'$ say.  If $\exists$ does not colour a single node of $H$, she loses by (C1), so suppose she colours $\set{h}$ where $h\in H$ and let the other node of $H$ be $h'$.    Since the map $\set{(g, h), (g', h')}$ is not an isomorphism, $\exists$ must lose by (C2).

\end{enumerate}
\end{proof} 


In the $2$-colour game and beyond, $\forall$ has much more power to force a win, and $\exists$'s play must satisfy many constraints if she intends to stave off defeat. We will describe some of these in Proposition \ref{P:constraints} below, but first we introduce some terminology and notation.

\begin{defn}[$\tau$, $\sigma$, $\sigma^+$]\label{D:tau} Let $G=(V, E)$ be a digraph with vertices $V$ and edges $E\subseteq V\times V$,  let $v$ be a vertex of $G$ and  $X, Y\subseteq V$.
\begin{align}
\omit\rlap{$\ind(v), \outd(v)$ are the in- and out-degrees of $v$}\\
\nonumber\ind(v)&=|E\cap(V\times\set v)| \\
\nonumber\outd(v)&=|E\cap(\set v\times V)|\\
\tau(v)&=(\ind(v),\outd(v)) \mbox{ is the \textbf{tally} of $v$}\\
\ind_Y(v)&=|E\cap (Y\times \set v)|\\
\outd_Y(v)&=|E\cap(\set v\times Y)| \\
\tau_Y(v)&=(\ind_Y(v), \outd_Y(v)) \mbox{, the tally of $v$ relative to $Y$}\\
\sigma(X)&=\set{\tau(x):x\in X}\\
\sigma_Y(X)&=\set{\tau_Y(x):x\in X}\\
 \omit\rlap{$\sigma^+(X)$ is the multiset of pairs $\tau(x)$ with $x\in X$}\\
\omit\rlap{$\sigma^+_Y(X)$ is the multiset of pairs $\tau_Y(x)$ with $x\in X$.}
\end{align}


\end{defn}

The definitions for $\tau$, $\sigma$ and $\sigma^+$ can be adapted for undirected graphs. Given a graph $G$ and $v\in G$, in this case $\tau(v)$ is just the degree of $v$, and $\sigma^+(G)$ is  essentially the degree sequence of $G$.

\begin{defn}
Given a subset $S$ of a graph $G$ define:
\begin{itemize}
\item $\eta_O(S) =S\cup \{v\in G: \exists u\in S\text{ and } (u,v) \text{ is an edge}\}.$
\item $\eta_I(S) = S\cup \{v\in G: \exists u\in S \text{ and } (v,u) \text{ is an edge}\}.$
\end{itemize}
If $\bar{s} = (s_1,\ldots,s_n)$ is a sequence where $s_i\in \{I,O\}$ for each $i$, we use $\eta_{\bar{s}}(S)$ to denote $\eta_{s_n}\circ \eta_{s_{n-1}}\circ\ldots \circ \eta_{s_1}(S)$. If $0<k<\omega$ we may write e.g. $\eta^k_O(S)$ for $\eta_{\bar{s}}(S)$ when $\bar{s}=(s_1,\ldots,s_k)$ and $s_i = O$ for all $i\leq k$. We can extend the notation to cover $k=0$ by setting  $\eta^0_O(S)=S$.
\end{defn}

\begin{prop}\label{P:constraints}
In $\G^2(G,H)$ with colours $\{red,blue\}$, if $\exists$ is pursuing a winning strategy her moves must satisfy the following constraints:
\begin{enumerate}[(S1)]
\item If $\forall$ colours a set $S$, and $T$ is the set coloured by $\exists$ in response, we must have $|S|=|T|$.
\item If $\forall$ makes a move colouring a single vertex $v$, then $\exists$ must respond by colouring a single vertex $w$ so that $\tau(v)=\tau(w)$. 
\item If $\forall$ colours a set of vertices $S$, then $\exists$ must respond by colouring a set $T$ such that $\sigma(S)=\sigma(T)$. 
\item If $\forall$ colours a set of vertices $S$, then $\exists$ must respond by colouring a set $T$ such that $\sigma^+(S)=\sigma^+(T)$.

\item If $\forall$ colours a set $S_0$ red, and $T_0$ is the set coloured red by $\exists$ in response, then if in his next move $\forall$ colours $S_1=\eta_O(S_0)$ blue, $\exists$ must colour $T_1=\eta_O(T_0)$ blue in response. The analogous result holds for $\eta_I$.

\item If $\forall$ colours a set $S$, and $T$ is the set coloured by $\exists$ in response, then whenever $\bar{s} = (s_1,\ldots,s_n)$ is a sequence where $s_i\in \{I,O\}$ for each $i$, we must have $|\eta_{\bar{s}}(S)|=|\eta_{\bar{s}}(T)|$.
\end{enumerate}
\begin{proof}\mbox{}
\begin{enumerate}[(S1)]
\item Suppose $\forall$ colours a set $S_0=S$ of vertices of $G$ red, and $\exists$ responds by colouring a set $T_0=T$ of vertices of $H$ red. Suppose without loss of generality that $n=|S_0|<|T_0|$. Then $\forall$ chooses arbitrary $v_0\in T_0$, defines $T_1 = T_0\setminus \{v_0\}$, and colours $T_1$ blue. Then $\exists$ must choose $S_1\subset S_0$ and colour it blue. Note that the inclusion must be strict, to avoid (C1). 

Now, $\forall$ continues by choosing $v_1\in T_1$, defining $T_2=T_1\setminus \{v_1\}$, and colouring $T_2$ red. Again, $\exists$ must respond by choosing $S_2\subset S_1$ and colouring it red. Repeating this process would produce a chain $T_0\supset T_1\supset T_2\supset\ldots \supset T_n$, where $|T_n|= |T_0|-n>0$. If $\forall$ has not won before this point, there would be a corresponding chain $S_0\supset S_1\supset S_2\supset\ldots \supset S_n$, but this is impossible, as $|S_n| \leq |S_0|-n =0$, and so $S_n$ is empty. Since $S_n$ is empty but $T_n$ is not, $\exists$ must lose.

\item Suppose $\forall$ colours the vertex $v$ of $G$ red, and $\exists$ responds by colouring the vertex $w$ of $H$ red (we know she must respond by colouring a single vertex). Without loss of generality, suppose $\ind(v)< \ind(w)$. Since if one of $\{v,w\}$ is reflexive and the other is not $\exists$ loses immediately, we can assume that they are either both reflexive, or both irreflexive. Let $S$ be the set of vertices of $G$ that have no outgoing edge to $v$. 

Suppose first that $S$ is empty. Then every vertex of $G$ has an outgoing edge into $v$. Since $\ind(v)< \ind(w)$ by assumption, it follows that $H$ has more vertices than $G$ (since we are assuming $v$ is reflexive if and only if $w$ is). Thus $\forall$ could win using (S1), and so $\exists$ could not pursue a winning strategy anyway.

Suppose then that $S\neq\emptyset$, and that $\forall$ colours $S$ blue. Then $\exists$ must respond by colouring a set $T$ of vertices of $H$ blue, and, by (S1), we must have $|S|=|T|$. But, as $\ind(v)<\ind(w)$, there must be a vertex $u$ in $T$ such that $(u,w)$ is an edge in $H$. Thus $\exists$ loses the game, as there is an edge from a blue vertex to a red vertex in $H$, but no such edge in $G$.
\item Suppose $\forall$ colours the set $S$ of vertices of $G$ red, and $\exists$ responds by colouring the set of vertices $T$ of $H$ the same colour. By (S1) we can assume $|S|=|T|$. Suppose $\sigma(S)\neq \sigma(T)$. Suppose without loss of generality that there is $v\in S$ such that $\tau(v)\neq \tau(w)$ for all $w\in T$. Suppose $\forall$ colours $v$ blue. Then $\exists$ must respond by colouring a vertex of $T$ blue, as otherwise there will be no vertex with palette $\{red,blue\}$ in $H$. Suppose $\exists$ colours the vertex $u\in T$ blue. Then $\tau(v)\neq\tau(u)$, by choice of $v$, and so $\exists$ will lose by  (S2).
\item Suppose $\forall$ colours the set $S$ of vertices of $G$ red, and $\exists$ responds by colouring the set of vertices $T$ of $H$ the same colour. By earlier work we can assume that $|T|=|S|$, and also that $\sigma(T)=\sigma(S)$. Suppose that $\sigma^+(T)\neq \sigma^+(S)$. Suppose without loss of generality that there is $(m,n)\in\omega^2$ such that there are strictly more vertices in $S$ whose tally is equal to $(m,n)$ than there are vertices in $T$ with that property. Let $S'=\{x\in S: \tau(x) = (m,n)\}$. Suppose $\forall$ colours $S'$ blue. Then $\exists$ must respond by colouring some subset $T'$ of $T$ blue, and she must have $|T'|=|S'|$. But as there are not enough vertices with the right tally in $T$, we will have $\sigma(T')\neq \sigma(S')$, and so the result follows from (S3).

\item Claim 1: $\exists$ must ensure that $T_1\supseteq \eta_O(T_0)$. Proof: note first that by (C1) we must have $T_0\subseteq T_1$. So,  if $\eta_O(T_0)\not\subseteq T_1$ then there will be a red-blue to uncoloured edge in $H$, while no such edges exist in $G$, as in $G$ all vertices connected to $T_0$ by outgoing edges are coloured blue. This triggers (C2), proving Claim 1.

Claim 2: Whenever $\forall$ colours a set $S$ and $\exists$ responds by colouring a set $T$, she must ensure that $|\eta_O(S)|=|\eta_O(T)|$.  Proof: without loss of generality, suppose $|\eta_O(S)|<|\eta_O(T)|$. Then $\forall$ can colour $\eta_O(S)$ blue, and by Claim 1, $\exists$ must colour in response a superset of  $\eta_O(T)$. Since the size of the superset must necessarily be greater than $|\eta_O(S)|$, this violates (S1). This proves Claim 2. 

Now, returning to the main proof, we established in Claim 1 that $\exists$ must ensure that $T_1\supseteq \eta_O(T_0)$. Also, by (S1) and Claim 2 we have must have $|T_1|=|S_1|=|\eta_O(S_0)|=|\eta_O(T_0)|$. So $T_1$ must be a superset of $\eta_O(T_0)$ with the same size. In other words, $T_1=\eta_O(T_0)$ as claimed. A similar argument works for $\eta_I$.

\item Suppose $\forall$ colours $S\subseteq G$ red and $\exists$ colours $T\subseteq H$ red in response. Given $\bar{s}=(s_1,\ldots,s_n)$,  suppose $\forall$ plays by colouring the sets $S=S_0, S_1, \ldots, S_k$ in sequence (alternating between red and blue appropriately), where $S_i = \eta_{s_i}(S_{i-1})$ for all $i\in\{1,\ldots,n\}$. Then $\exists$ must respond by colouring a sequence $T=T_0, T_1, \ldots, T_k$.  By induction and (S5) we must have $T_i = \eta_{s_i}(T_{i-1})$ for all $i\in\{1,\ldots,n\}$, so if $|\eta_{\bar{s}}(S)|\neq |\eta_{\bar{s}}(T)|$ she loses by (S1). 
\end{enumerate}
\end{proof}
\end{prop}  

Note that it follows easily from Proposition \ref{P:constraints}(S1) that, if $|G|\neq |H|$, then $\forall$ has a winning strategy in $\G^2(G,H)$. (S6) can be improved to take tallies into account, but we postpone this argument till Corollary \ref{C:neighbours}, where we will prove something stronger. We also have the following easy lemma giving us a rough and ready upper bound on the number of colours $\forall$ needs to guarantee a win.

\begin{lem}\label{L:rough_bound}
Let $G$ and $H$ be digraphs with $|G|=|H| = n\geq 2$ but $G\not\cong H$. Then $\forall$ has a winning strategy in  $\G^{\lceil \log_2 n\rceil}(G,H)$.   
\end{lem} 
\begin{proof}
 With  $\lceil \log_2 n\rceil$ colours $\forall$ can give each vertex in $G$ a unique palette, and to avoid losing $\exists$ will end up colouring $H$ so that each of its vertices has a unique palette. But now she loses anyway, because $G\not\cong H$.   
\end{proof}

In the 3-colour game we can get a version of constraint (S4) from Proposition \ref{P:constraints} for $\sigma^+_S(X)$, as we make precise in the following result.
\begin{prop}
In the 3-colour game played over a pair of digraphs $(G,H)$ with colours $\{red,blue,green\}$, if $\exists$ is pursuing a winning strategy then, at any stage in the game, if $S$ and $X$ are the subsets of $G$ coloured red and blue, respectively, and $T, Y$ are the subsets of $H$ coloured red and blue, respectively, then we must have  $\sigma^+_S(X) = \sigma^+_T(Y)$, and $\sigma^+_X(S) = \sigma^+_Y(T)$. 
\end{prop}   
\begin{proof}
The idea is to copy the format of the proof of Proposition \ref{P:constraints}, by proving analogues of (S2), (S3), building up to the analogue of (S4) that is the statement that $\exists$ must ensure that $\sigma^+_S(X) = \sigma^+_T(Y)$. By symmetry it follows she must also ensure that $\sigma^+_X(S) = \sigma^+_Y(T)$. The strategy in each case is also  essentially the same as in the proof of Proposition \ref{P:constraints}, just using the third colour to `relativize'. To illustrate the technique, we provide the proof that $\exists$ must ensure $\tau_S(u) = \tau_T(v)$. We leave the rest to the reader.

So, suppose that at some point in the game $S$ and $\{u\}$ are coloured, respectively, red and blue in $G$, and $T$ and $\{v\}$ are coloured red and blue in $H$. Suppose that $\tau_S(u)\neq \tau_T(v)$, and suppose without loss of generality that 
\[|\{s\in S: (s,u) \text{ is an edge in } G\}|<|\{s\in S: (x,v) \text{ is an edge in } H\}|.\] 
We can also assume that $u$ is reflexive if and only if $v$ is. Let $S'$ be the set of vertices in $S$ that do not have an outgoing edge to $u$. If $S'$ were empty, every vertex in $S$ would have an outgoing edge to $u$. By the assumed inequality, it would follow that $|S|<|T|$, and thus $\exists$ would not be following a winning strategy, by (S1).

Suppose then that $S'$ is not empty, and that $\forall$ colours $S'$ green. Then $\exists$ must respond by colouring a subset $T'$ of $T$ with $|T'|=|S'|$ green, and from elementary cardinality considerations it follows that there is a red/green to blue edge in $H$, but not in $G$.  
\end{proof}

\section{Tally-sequences}\label{S:tally}
In this section we introduce the concept of a tally-sequence. This provides a means to partition the vertices of a graph. The approach is similar to that of colour refinement, but the induced partition is coarser. The advantage of tally sequences is that $\exists$ must preserve them in her responses as part of any winning strategy in $\G^2$ (see Corollary \ref{C:2-colours}), but it is not known whether the same is true for colour refinement colours. We will exploit the fact that $\exists$ must preserve tally-sequences many times in our results on Stockmeyer graphs in Section \ref{S:games}. 
\begin{defn}\label{D:t-sequence}
Given a digraph $G$, a vertex $v$ of $G$, and a set $X$ of vertices of $G$, we define for each vertex $v$ of $G$ the sequence $(t_0^v, t_1^v, \ldots)$ of pairs of natural numbers recursively as follows:
\[ t_0^v=\tau_X(v).\]
If $0\leq k$ and $(t_0^v, t_1^v, \ldots, t_k^v)$ has been defined for every vertex $v$ of $G$ then
\[ t_{k+1}^v=\tau_{X_k^v}(v) \mbox{ where } X^v_k=\set{u\in X: (t^u_0, t^u_1, \ldots, t^u_k)=(t^v_0, t^v_1, \ldots, t^v_k)}.\]
In other words, $t^v_{k+1}$ is the tally of $v$ relative to the set of vertices $u$ in $X$ for which $(t^u_0,\ldots, t^u_k)$ is the same as $(t^v_0, \ldots, t^v_k)$.  For every $n<\omega$, we use $\vec{\tau}^n_X(v)$ to denote the sequence  $(t_0^v, \ldots, t_n^v)$, and $\vec{\tau}_X(v)$  denotes the sequence $(t_0^v, t_1^v, \ldots)$ described above.     We call $\vec{\tau}_X(v)$ the \textbf{tally-sequence of $v$ relative to $X$}, and if $X$ is the set of all vertices of $G$ we just write $\vec{\tau}(v)$ and speak of the \textbf{tally-sequence} of $v$.
\end{defn} 

As with Definition \ref{D:tau}, the concept of a tally-sequence can be trivially adapted for undirected graphs. For undirected graphs, the concept is very similar to that underlying \textbf{colour refinement}, also known as the Weisfeiler-Leman algorithm in one dimension (see e.g. \cite[Section 3.5.1]{Gro17}), but it is not the same. The $k$-dimensional WL algorithm is known to be closely associated with the logic $\mathsf{C}^{k+1}$, which is first-order logic restricted to $k+1$ variable symbols but extended by so-called counting quantifiers (see e.g. \cite[Section 3.4.2]{Gro17}). Indeed, it is known that two relational structures are indistinguishable by the $k$-WL algorithm if and only if they agree about all $\mathsf{C}^{k+1}$ sentences of the corresponding signature (\cite{ImmLan90}, or see e.g. \cite[Theorem 3.5.7]{Gro17}). It is easy to show that if two vertices of an undirected graph $G$ are assigned the same colours by colour refinement, then they also have the same tally sequences. The converse does not hold in general, so the colour refinement colours provide a more refined partition of vertices than the equivalence classes induced by tally sequences. This relative lack of refinement surprisingly turns out to be an advantage for finding winning strategies for $\forall$ in Seurat games, which is why we define tally sequences rather than simply adapting the well known colour refinement process to directed graphs. This point will be elaborated after Lemma \ref{L:2-colours}, but the basic idea is that in a winning strategy in the Seurat game with 2 colours, $\exists$ must in her moves match the tally sequences of vertices involved in the move made by $\forall$.  

As with colour refinement colour, the tally-sequence of a vertex eventually stabilizes. To see this, observe that in the construction of $\vec{\tau}_X(v)$, if $X^v_i = X^v_{i+1}$ for some $i$ then, $t_{k} = t_{i+1}$ for all $k> i$, and $X^v_k = X^v_i$ for all $k\geq i$. As $X^v_{i+1}\subseteq X^v_i$ for all $i<\omega$, the sequence must therefore stabilize eventually, as $G$ is finite. So the size of the graph $G$ induces an upper bound on the number of elements in the interesting parts of the tally-sequences of its vertices. We call the initial part of a tally-sequence $\vec{\tau}_X(v)$ before repetition begins the \textbf{significant part} of the tally-sequence. 
\begin{defn}
Define the \textbf{tally-spectrum} of a digraph $G$ to be the multiset of tally-sequences for its vertices. If $X$ is a set of vertices of $G$, define the \textbf{tally-spectrum of $X$ in $G$} to be the multiset of tally-sequences of vertices in $X$ relative to $G$.
\end{defn}

A naive algorithm for computing the (significant part of) the tally spectra of a graph $G$ can easily be shown to run in polynomial time in the number of vertices. We note that the $k$-dimensional  Weisfeiler-Leman algorithm can be implemented in $O(n^{k+1}\log n)$ time, where $n$ is the number of vertices \cite{ImmSen19}.  The following trivial proposition holds, since all graph properties defined by the edges and vertices  are preserved under isomorphism.

\begin{prop}\label{P:tally-isom}
Suppose $f:G\to H$ is an isomorphism. Then, for all vertices $v$ of $G$ we have $\vec{\tau}(v) = \vec{\tau}(f(v))$.
\end{prop}

The next lemma says, essentially, that the 2-colour game can `see' tally-sequences. In particular, if $\forall$ colours a single element, $\exists$ must respond by colouring an element with the same tally-sequence. This will be extremely useful in Section \ref{S:games}.
\begin{lem}\label{L:2-colours}
Suppose $G$ and $H$ are digraphs, and let $n<\omega$. Suppose in $\G^2(G,H)$ that $\forall$ colours a set $X$ of vertices of $G$ red. Suppose there is a sequence $s = (s_0,s_1,\ldots, s_n)$ of pairs of natural numbers such that for all $x\in X$ we have $\vec{\tau}^n(x) = (s_0,\ldots,s_n)$. Then, as part of a winning strategy, $\exists$ must respond by colouring a set $Y$ with the same size as $X$, and where, for all $y\in Y$, we have $\vec{\tau}^n(y) = (s_0,\ldots,s_n)$. 
\end{lem}
\begin{proof}
That $Y$ must be the same size as $X$ is Proposition \ref{P:constraints}(S1). We will prove the rest by induction on $n$. The base case follows immediately from Proposition \ref{P:constraints}(S3). For the inductive step, suppose the result is true for $n$. Define $G'$ to be the subgraph generated by set of all vertices $v$ of $G$ such that $\vec{\tau}^n(v) = (s_0,\ldots,s_n)$, and define $H'$ analogously. By the inductive hypothesis, if $\forall$ restricts his play to $G'$ and $H'$, then so too must $\exists$, if she does not want to lose. Suppose $\forall$ colours a set $X$ of vertices of $G$ red, and suppose that for all $x\in X$ we have $\vec{\tau}^{n+1}(x) = (s_0,\ldots,s_{n+1})$. Then, by the above considerations, without loss of generality, we can consider this to be a move in the 2-colour game played over $(G',H')$. 

So, suppose $\exists$ colours the set $Y$ of vertices of $H'$ red in response, and suppose also there is $u\in Y$ with $\vec{\tau}^{n+1}(u) = (s_0,\ldots,s_n, s')$, and $s'\neq s_{n+1}$. Note that to say that $\vec{\tau}^{n+1}(u) = (s_0,\ldots,s_n, s')$ is to say that ${\tau}_{H'}(u) = s'$.  Then $\forall$ can continue by colouring $u$ blue. Now, $\exists$ must respond by colouring some vertex $w$ of $X$ blue. By choice of $X$ we have ${\tau}_{G'}(w) = s_{n+1}$, but, for this to be part of a winning strategy for $\exists$ in $\G^2(G',H')$, she must ensure ${\tau}_{G'}(w) = s'$. Since $s_{n+1}\neq s'$, by assumption, $\exists$ must lose $\G^2(G',H')$, and thus $\G^2(G,H)$ too.
\end{proof}
Note that it is crucial in the proof above that when computing the next step in the tally sequence of a vertex, only the vertices whose tally sequences are equal up to that point are involved. This allows the inductive step to go through as described. It follows that this argument does not work for colour refinement colours, and we do not know if the analog of Lemma \ref{L:2-colours} for colour refinement colours holds. Lemma \ref{L:2-colours} can be strengthened as follows.

\begin{cor}\label{C:2-colours}
In $\G^2(G,H)$, if $\forall$ colours a subset $X$ of $G$ red then $\exists$ must respond by colouring a subset $Y$ of $H$ red, and the tally-spectrum of $Y$ in $H$ must be the same as the tally-spectrum of $X$ in $G$. 
\end{cor}
\begin{proof}
Suppose  $\exists$ colours $Y$ and the tally-spectrum of $Y$ in $H$ is not the same as that of $X$ in $G$. Then there is a sequence $s=(s_0,s_1,\ldots)$ of pairs of natural numbers such that the number of elements of $X$ whose tally-sequence is $s$ is not the same as the number of elements of $Y$ whose tally-sequence is $s$. Without loss of generality, suppose there are more elements of $X$ with this tally-sequence. Then $\forall$ can colour these elements blue, and, by Lemma \ref{L:2-colours}, $\exists$ must, if she doesn't want to lose, respond by colouring the same number of elements of $Y$ with tally-sequence $s$ blue, but this is impossible.  
\end{proof}

\begin{cor}\label{C:distinct}
If digraphs $G$ and $H$ do not have the same tally-spectra then $\forall$ has a winning strategy in $\G^2(G,H)$.
\end{cor}
\begin{proof}
This follows immediately from Corollary \ref{C:2-colours}.
\end{proof}

We can also generalize Proposition \ref{P:constraints} (S6).

\begin{cor}\label{C:neighbours}
In $\G^2(G,H)$, if $\forall$ colours a set $S$, and $T$ is the set coloured by $\exists$, then whenever $\bar{s} = (s_1,\ldots,s_n)$ is a sequence where $s_i\in \{I,O\}$ for each $i$, the tally spectra of $\eta_{\bar{s}}(S)$ and $\eta_{\bar{s}}(T)$ must be the same, or else $\forall$ can force a win.
\end{cor}
\begin{proof}
Suppose that after $S$ and $T$ are coloured by $\forall$ and $\exists$ respectively, $\forall$ colours the sets $S_1,\ldots,S_n$, alternating between red and blue appropriately, where $S_1=\eta_{s_1}(S)$ and $S_i = \eta_{s_i}(S_{i-1})$ for all $i>1$. Then $\exists$ must respond by colouring sets $T_1,\ldots,T_n$. By induction and Proposition \ref{P:constraints} (S5), $\exists$ must play $T_1 = \eta_{s_1}(T)$, and $T_i = \eta_{s_i}(T_{i-1})$ for $i>1$, otherwise $\forall$ can win. So $S_n=\eta_{\bar{s}}(S)$, and $T_n$ must be $\eta_{\bar{s}}(T)$. Thus if these sets have different tally spectra, we know from Corollary \ref{C:2-colours} that  $\forall$ can force a win. 
\end{proof}

As discussed above, the proof of Lemma \ref{L:2-colours} does not work for colour refinement colours, and we do not know if $\exists$ has to match a set of nodes  all with some colour refinement, with a corresponding set of the same size with the same colour refinement, in order to survive the game.  
 If this were true, we could then prove:
 \begin{itemize}
\item[(*)] two graphs distinguishable by colour refinement can be distinguished by the Seurat game $\G^2$.
\end{itemize}  
Noting the fact that graphs $G$ and $H$ are distinguishable by colour refinement if and only if there is a sentence $\phi$ in the logic $\mathsf{C}^2$ (first-order logic restricted to 2 variable symbols but extended by counting quantifiers) with $G\models \phi$ and $H\not\models \phi$ (\cite{ImmLan90}, or see e.g. \cite[Section 5]{CFI92} or \cite[Theorem 3.5.5]{Gro17}), we would then have the following sequence of implications for graphs $G$ and $H$:
\begin{align*}
&\text{There exists a $\mathsf{C}^2$-sentence $\phi$ such that $G\models \phi$ and $H\not\models \phi$}\\
\iff& \text{colour refinement distinguishes $G$ and $H$}\\
\implies& \text{$\forall$ has a winning strategy in $\G^2(G,H)$  (by assumption (*))}\\
\implies& \text{$\forall$ has a winning strategy in $\M^2_2(G,H)$ (by Proposition \ref{P:MSO})}\\
\iff& \text{There is a 2nd-order sentence $\psi$ with up to 2 monadic and 2 first-order}\\ 
\phantom{\implies}&\text{ variables such that $G\models \psi$ and $H\not\models \psi$.}  
\end{align*}
A pebble game capturing $\mathsf{C}^2$ equivalence is described in \cite[Section 4.1]{CFI92} (see also \cite[Fact 3.4.15]{Gro17}), so assumption (*) is equivalent to saying that whenever $\forall$ has a winning strategy in this pebble game for graphs $G$ and $H$, he also has one in $\G^2(G,H)$. We do not see why this should be true, but we do not have a proof that it is not. We note that finding graphs $G$ and $H$ that are not $\mathsf{C}^2$-equivalent but such that $\forall$ does not have a winning strategy in $\G^2(G,H)$ seems difficult, as currently we do not know of \emph{any} non-isomorphic graphs where $\forall$ does not have a winning strategy in the 2-colour Seurat game. 

Returning to tally-spectra, we have the following obvious result.

\begin{cor}\label{C:iso-tally}
If $G$ and $H$ are isomorphic they must have the same tally-spectra.
\end{cor}
This Corollary  follows from Corollary \ref{C:distinct},  but holds trivially since all graph properties defined by edges and vertices are preserved by isomorphism.

The converse to Corollary \ref{C:iso-tally} does not hold, as demonstrated by the following example, which is also well known as an example of a situation where the colour refinement algorithm fails to distinguish non-isomorphic graphs. However, $\forall$ does have a strategy in the corresponding 2-colour Seurat game.
  
	\begin{ex}\label{E:tally1}
Given a pair of non-isomorphic regular graphs, if the vertices of both graphs have the same in- and out-degrees, then the pair cannot be distinguished by looking at their tally-spectra. However, we do not know whether $\forall$ must have a winning strategy in the 2-colour Seurat game over such a pair.   We provide an example where $\forall$ can win the 2-colour Seurat game over the pair.  Let $G$ and $H$ be the (undirected) graphs in Figures \ref{F:star1} and \ref{F:star2} respectively. Then the tally-sequence of the central vertex in both graphs is $(6,0,0,\ldots)$, and the tally-sequences of the other vertices are all $(3,2,2,\ldots)$. However, the graphs are not isomorphic, because $G$ contains a cycle of length 6 not passing through the central vertex, but $H$ does not. 
	
	Moreover, $\forall$ has a strategy in  $\G^2(G,H)$, because he can colour one of the 3-cycles of exterior vertices of $H$ red, and the other blue. $\exists$ must lose, as in $H$ there will be no edge connecting red and blue, but if she follows the necessary principles of winning play by matching set sizes (see Proposition \ref{P:constraints}(S1)), there will inevitably be such an edge in $G$. 
	\end{ex}
	
	\begin{figure}[ht]
  \centering
  \begin{minipage}[b]{0.49\textwidth}
\[\xymatrix{& \bullet\ar@{-}[rr] & & \bullet\ar@{-}[dr] &\\
\bullet\ar@{-}[rr]\ar@{-}[ur] & & \bullet\ar@{-}[rr]\ar@{-}[ur]\ar@{-}[ul]\ar@{-}[dr]\ar@{-}[dl] & & \bullet\\
& \bullet\ar@{-}[ul]\ar@{-}[rr] & & \bullet\ar@{-}[ur] &
}\]
\caption{$G$}
\label{F:star1} 
  \end{minipage}
  \hfill
  \begin{minipage}[b]{0.49\textwidth}
\[\xymatrix{& \bullet\ar@{-}[dd]\ar@{-}[drrr]   & & \bullet\ar@{-}[dd]\ar@{-}[dlll]   &\\
\bullet & & \bullet\ar@{-}[rr]\ar@{-}[ur]\ar@{-}[ul]\ar@{-}[dr]\ar@{-}[dl]\ar@{-}[ll]  & & \bullet\\
& \bullet\ar@{-}[urrr]  & & \bullet\ar@{-}[ulll]  &
}\]
\caption{$H$}
\label{F:star2} 
  \end{minipage}
\end{figure}

In Example \ref{E:tally1}, we notice that something that makes $G$ different from $H$ is that the subgraph of $G$ composed of vertices whose tally-sequence is $(3,2,2,\ldots)$ is connected, being isomorphic to the 6-cycle $C_6$, while the corresponding subgraph of $H$ is the disjoint union of two copies of $C_3$, and so is not. We might wonder if we could obtain a kind of converse to Corollary \ref{C:iso-tally} by ruling out this kind of counterexample by, for example, demanding subgraphs induced by tally-sequences also be isomorphic. Unfortunately, this doesn't work, as we demonstrate in Example \ref{E:tally2}. First we will make a definition to clarify the idea of a `subgraph induced by a tally-sequence'.

\begin{defn}
Let $G$ be a digraph, let $n<\omega$ and let $s = (s_0,s_1,\ldots)$ be a sequence of ordered pairs of natural numbers. Define $G^n_s$ to be the subgraph of $G$ induced by the set of vertices $v$ of $G$ such that $\vec{\tau}^n(v) = (s_0,\ldots,s_n)$.
\end{defn}  

\begin{ex}\label{E:tally2}  
We present two non-isomorphic digraphs $G$ and $H$ that have the same tally-spectra, and where, in addition, for each tally-sequence $s$ and for each $n<\omega$ the induced subgraphs $G^n_s$ and $H^n_s$ are isomorphic. This example occurs in \cite{Rama81}, in a slightly different context, just before Theorem 2. Let $G$ and $H$ be the tournaments described by the adjacency matrices in Figures \ref{F:rama1} and \ref{F:rama2} respectively. The significant parts of their tally-spectra are given in Figures \ref{F:tally-rama1} and \ref{F:tally-rama2}. 

Note that $\forall$ still has a strategy in the 2-colour game played over these graphs. This is because he can colour $v_5$ red, in which case $\exists$ must respond by colouring $w_3$ red, to match tally-sequences. Then $\forall$ can colour $v_4$ blue, and $\exists$ must respond by colouring $w_4$ blue, for the same reason. But now there is a blue to red edge in $G$, but no such edge in $H$, and so $\exists$ loses anyway.

The graphs in this example are a counterexample to the reconstruction conjecture for digraphs (see \cite{Rama81}). Thus we see that, like the Stockmeyer graphs to be discussed later, we have a pair of digraphs that cannot be distinguished by comparing decks, but \emph{can} be distinguished in the Seurat game with only two colours. 
\end{ex}

\begin{figure}[h]
  \centering
  \begin{minipage}[b]{0.49\textwidth}
 \begin{tabular}{l  | l l l l l l}
& $v_0$ & $v_1$ & $v_2$ & $v_3$ & $v_4$ & $v_5$ \\ \hline
$v_0$ & 0 & 1 & 1 & 1 & 1 & 0 \\
$v_1$ & 0 & 0 & 1 & 1 & 1 & 0 \\
$v_2$ & 0 & 0 & 0 & 1 & 1 & 1 \\
$v_3$ & 0 & 0 & 0 & 0 & 1 & 1 \\
$v_4$ & 0 & 0 & 0 & 0 & 0 & 1 \\
$v_5$ & 1 & 1 & 0 & 0 & 0 & 0 
\end{tabular}
\caption{The adjacency matrix of $G$}
\label{F:rama1} 
  \end{minipage}
  \hfill
  \begin{minipage}[b]{0.49\textwidth}
  \begin{tabular}{l  | l l l l l l}
& $w_0$ & $w_1$ & $w_2$ & $w_3$ & $w_4$ & $w_5$ \\ \hline
$w_0$ & 0 & 1 & 1 & 1 & 1 & 0 \\
$w_1$ & 0 & 0 & 1 & 1 & 1 & 0 \\
$w_2$ & 0 & 0 & 0 & 1 & 1 & 0 \\
$w_3$ & 0 & 0 & 0 & 0 & 1 & 1 \\
$w_4$ & 0 & 0 & 0 & 0 & 0 & 1 \\
$w_5$ & 1 & 1 & 1 & 0 & 0 & 0 
\end{tabular}
\caption{The adjacency matrix of $H$}
\label{F:rama2} 
  \end{minipage}
	\end{figure}
	
	\begin{figure}[h]
  \centering
  \begin{minipage}[b]{0.49\textwidth}
 \begin{tabular}{l  | l l l}
vertex &\multicolumn{3}{c} {sig. part of $\vec{\tau}(\mbox{vertex})$} \\ \hline
$v_0$ & ((1,4),  & (0,0))  &  \\
$v_1$ & ((2,3),  & (0,1),  & (0,0)) \\
$v_2$ & ((2,3),  & (1,0),  & (0,0))\\
$v_3$ & ((3,2),  & (0,1),  & (0,0))\\
$v_4$ & ((4,1),  & (0,0))  & \\
$v_5$ & ((3,2),  & (1,0),  & (0,0))
\end{tabular}
\caption{The tally-spectrum of $G$}
\label{F:tally-rama1} 
  \end{minipage}
  \hfill
  \begin{minipage}[b]{0.49\textwidth}
  \begin{tabular}{l  | l l l}
vertex  &\multicolumn{3}{c} {sig. part of $\vec{\tau}(\mbox{vertex})$}  \\ \hline
$w_0$ & ((1,4),  & (0,0))  & \\
$w_1$ & ((2,3),  & (1,0),  & (0,0)) \\
$w_2$ & ((3,2),  & (0,1),  & (0,0)) \\
$w_3$ & ((3,2),  & (1,0),  & (0,0)) \\
$w_4$ & ((4,1),  & (0,0))  & \\
$w_5$ & ((2,3),  & (0,1),  & (0,0)) 
\end{tabular}
\caption{The tally-spectrum of $H$}
\label{F:tally-rama2} 
  \end{minipage}
	\end{figure}

\section{Seurat games and the $k$-WL algorithm}\label{WL}	
In this section we describe the construction of pairs of non-isomorphic graphs which are not distinguished by the $k$-dimensional Weisfeiler-Leman algorithm, for $k< \omega$. This family was introduced, and the result about $k$-WL was proved, in \cite[Section 6]{CFI92}. The new result here is that all these pairs of graphs \emph{are} distinguishable in the Seurat game $\G^2$. A diagram of one of these constructions can be found in \cite[p85]{Gro17}, and we provide the formal details now.

Given a graph $G$, let $E(G)$ denote the set of edges of $G$. We will assume here that $G$ is undirected and  irreflexive, and that every node has degree at least one. If $v$ is a node of $G$, we will use $d(v)$ to denote the degree of $v$. Given a node $v\in G$, the gadget $X_v$ is the graph with nodes \[\set{i(v, S):  S\subseteq\set{(v, w): (v, w)\in E(G)},\;|S|\mbox{ is even}}\;\cup\;\set{a(v, w), b(v, w): (v, w)\in E(G)}.\] 
Note that here $i(v,S),a(v,w),b(v,w)$ are names nodes, and $i,a,b$ are not meant to be understood as functions. The left set contains $2^{d(v)-1}$ nodes, and the right set contains $2d(v)$ nodes. Nodes of the former type are called \textbf{internal}, those of the latter type are called \textbf{external}.  The set of edges of the gadget $X_v$ is 
\[\set{(i(v, S), a(v, w)): (v, w)\in S}\cup\set{((i(v, S), b(v, w)): (v, w)\in E(G)\setminus S}. \]
Thus each internal node is linked to exactly  half the externals of its gadget, and similarly each external is linked to exactly half of the internals in its gadget.  The graph $\Gamma(G)$ is obtained by taking the disjoint union of all gadgets $X_v$ where $v\in G$, and for each $v\neq w\in G$, adding edges \[\set{(a(v, w), a(w, v)), (b(v, w), b(w, v)):(v, w)\in E(G)}\] between gadgets.

$\widetilde{\Gamma(G)}$ is obtained from $\Gamma(G)$ by picking an arbitrary edge $(v, w)$ and replacing the two edges \[\set{(a(v, w), a(w, v)), (b(v, w), b(w, v))}\] by the twist \[\set{(a(v, w), b(w, v)), (b(v, w), a(w, v))},\] so node degrees are unchanged.   By \cite[Lemma 6.2]{CFI92},  provided  all nodes have degree at least two, up to isomorphism it does not matter which edge $(v, w)$ is chosen.

As in \cite[Definition 6.3]{CFI92}, we define a \textbf{separator} of $G$ to be a subset $S\subset G$ such that the subgraph of $G$ induced by deleting the nodes in $S$ has no connected component with more than $\frac{|G|}2$ vertices. 

\begin{thm}
Given $n<\omega$, let $K_n$ be the complete graph with $n$ vertices. Then:
\begin{enumerate}[(1)]
\item If $n\geq 3$ then  $\Gamma(K_n)$ is not isomorphic to $\widetilde{\Gamma(K_n)}$ 
\item If $n\geq 1$, then $\Gamma(K_{2n+1})$ cannot be distinguished from $\widetilde{\Gamma(K_{2n+1})}$ by the $n$-dimensional WL algorithm.
\item If $n\geq 4$, then $\forall$ has a winning strategy in $\G^2(\Gamma(K_{n}), \widetilde{\Gamma(K_{n})})$.  
\end{enumerate}
\end{thm}
\begin{proof}
 (1) follows immediately from  \cite[Lemma 6.2]{CFI92}. (2) follows from the proof of \cite[Theorem 6.4]{CFI92}, as $K_{2n+1}$ contains no separator with fewer than $n+1$ vertices. For (3), we describe a winning strategy for $\forall$ as follows.

Note first that, according to the definition, each gadget in $\Gamma(K_{n})$ and $\widetilde{\Gamma(K_{n})}$ has $2^{n-2}$ internal, and $2(n-1)$ external nodes. Each internal node thus has degree $n-1$, and each external node has degree $2^{n-3}+1$. Now, in the initial round, $\forall$ colours 
\[R = \set{i(v, \emptyset):v\in K_n} \;\subset\Gamma(K_{n})\] 
red. Then $\exists$ must respond by colouring red exactly one internal node from each gadget of $\widetilde{\Gamma(K_{n})}$, as we now demonstrate. 

First, denote the set coloured by $\exists$ in her response by $\tilde{R}$. As $n\geq 4$, by consideration of degrees and Proposition \ref{P:constraints} (S4), $|\tilde{R}|=n$ (as $|R|=n$), and $\tilde{R}$ can contain only internal nodes. Suppose $\tilde{R}$ contains two nodes from the same gadget. Then there must be another gadget of $\widetilde{\Gamma(K_{n})}$ where no internal nodes are coloured. Suppose $\forall$ colours the set of internal nodes of this gadget blue (call this set $\tilde{B}'$). Then $|\tilde{B}'|=2^{n-2}$, and so $\exists$ must respond by colouring a set $B'$ of $2^{n-2}$ previously uncoloured internal nodes of $\Gamma(K_{n})$. As exactly one node of each gadget of $\Gamma(K_{n})$ is coloured red, $\exists$ can't let $B'$ be the internal nodes of a single gadget. Now, $\eta(\tilde{B}')$ is the union of $\tilde{B}'$ with the external nodes of its gadget. So $|\eta(\tilde{B}')|= 2^{n-2}+ 2(n-1)$. Moreover, $2^{n-2}\geq n$, as $n\geq 4$, so $|B'| \geq n$, and so $\exists$ must colour at least as many internal nodes as there are gadgets. Any distinct pair of internal nodes from the same gadget must have at least $n$ neighbours, because it is impossible for two internal nodes to share all of the same external neighbours, and each internal node is neighbour to exactly $n-1$ external nodes ---  half the external nodes in its gadget. Any pair of internal nodes from different gadgets must have $2(n-1)$ distinct neighbours. So $\exists$ cannot hope to make $|\eta(B')|$ less than $2^{n-2}+n+ (n-1) > 2^{n-2} + 2(n-1)$. I.e. $|\eta(B')|>|\eta(\tilde{B}')|$, conflicting with Proposition \ref{P:constraints} (S6). Thus, if $\exists$ does not colour exactly one internal node from each gadget in response to $\forall$'s opening move, $\forall$ can force a win.

Returning to the main game, $\forall$ proceeds by colouring 
\[B = \set{a(v, w):v\neq w\in K_{n}}\subset\Gamma(K_{n})\] 
blue. In other words, $\forall$ colours every $a$ type external node of every gadget in $\Gamma(K_{n})$.  Then $\exists$ must respond by colouring $\tilde{B}\subset \widetilde{\Gamma(K_{n})}$, and $\tilde{B}$ must consist of exactly one node from each pair $({a}(v, w),{b}(v, w))$ for each $v\neq w\in K_{n}$. To see why, note first that Proposition \ref{P:constraints} (S4) says she must colour exactly $n(n-1)$ external nodes.  If she  coloured both nodes in a pair $({a}(v, w),{b}(v, w))$,  then,  as every internal node in a gadget is the neighbour of either the $a$ or the $b$ node in each external pair of that gadget, this would result in there being a blue---red edge in $\widetilde{\Gamma(K_{n})}$. Since no such edge exists in $\Gamma(K_{n})$, this would trigger trigger (C1).   So we may assume she colours blue exactly one out of each pair  $({a}(v, w),{b}(v, w))$ in $\tilg$.

We can assume without loss of generality that $\tilde{R} = \set{i(v, \emptyset):v\in K_n}$, and that $\tilde{B} = \set{a(v, w):v\neq w\in K_{n}}$. This amounts to assuming the internal node coloured by $\exists$ is $i(v, \emptyset)$ in each gadget, and switching the labels within each $(a,b)$ external node pair appropriately. This assumption is safe because the graph obtained by adding an even number of twists to $\widetilde{\Gamma(K_{n})}$ is isomorphic to $\widetilde{\Gamma(K_{n})}$, via an isomorphism induced by the relabeling described here (see \cite[Lemma 6.2]{CFI92}).

In the third round, $\forall$ reuses red to colour all the internal nodes of $\Gamma(K_{n})$, and in response $\exists$ must colour red all the internal nodes on $\widetilde{\Gamma(K_{n})}$ (appealing to Proposition \ref{P:constraints} (S4)). Now, in $\Gamma(K_{n})$, the (external) nodes in $B$ have as neighbours only the red coloured internal nodes of their own gadget, and a blue coloured external node from a neighboring gadget. In particular, there are no blue---uncoloured edges in $\Gamma(K_{n})$. However, due to the twist, there is an edge like this in $\widetilde{\Gamma(K_{n})}$, so $\exists$ loses anyway.     
\end{proof}

The result above produces, for each $k\geq 4$, a pair of graphs that cannot be distinguished by the $k$-WL algorithm, but can be distinguished in $\G^2$.
We do not currently know whether there are non-isomorphic graphs that can be distinguished by $k$-WL, but not by $\G^2$, or even if there are \emph{any} non-isomorphic graphs that $\G^2$ cannot distinguish.

\section{Stockmeyer graphs}\label{S:stock}
In this section we describe some graph constructions of P.K. Stockmeyer. The constructions come in pairs, and were originally used to demonstrate that there are non-isomorphic directed graphs that cannot be distinguished comparing decks. The reason we describe them here is that in the next section we will prove that they all \emph{can} be distinguished by $\G^2$.  We use the approach taken in \cite{Stock81}, with some minor notational differences. Given $0<k<\omega$, we define the tournament $T_k$ with vertices $\{v_1,\ldots,v_{2^k}\}$. The edge relation of $T_k$ is defined by there being an edge from $v_i$ to $v_j$ if and only if $\odd(j-i) \equiv 1 \mod 4$, where $\odd(z)$ is the result of dividing an integer $z$ by the largest possible power of 2.  Note that  $T_k$ is irreflexive and for distinct vertices $i, j$ exactly one of $(i, j)$ and $(j, i)$ is an edge,  hence it defines  a tournament.  

Following Stockmeyer, we will describe six families of pairs of graphs. The graphs involved will be disjoint unions of $T_m$ and $T_n$, for $0\leq n < m$, with additional edges between $T_m$ and $T_n$ defined according to a system to be described shortly. We will need the following definition.

\begin{defn}[$M_{m,n}$]
For natural numbers $0\leq n < m$, let $p = 2^m+2^n$, and define the $p\times p$ matrix $M_{m,n}$ as follows ($w,x,y,z$ will appear as variables). 
\begin{itemize}
\item If $1\leq i,j \leq 2^m$, or if $2^m+1\leq i,j \leq p$, set $M_{m,n}[i,j] = 1$ if $\odd(j-i) \equiv 1 \mod 4$, and $M_{m,n}[i,j] = 0$ otherwise.
\item If $1\leq i \leq 2^m$ and $2^m+1\leq j \leq p$, set $M_{m,n}[i,j] = w$ if $i+j$ is even, and $M_{m,n}[i,j] = x$ otherwise.
\item If $1\leq j \leq 2^m$ and $2^m+1\leq i \leq p$, set $M_{m,n}[i,j] = y$ if $i+j$ is even, and $M_{m,n}[i,j] = z$ otherwise.
\end{itemize}
\end{defn} 
Thus $M_{m,n}$ contains a copy of the adjacency matrix of $T_m$ in its upper left part, a copy of the adjacency matrix of $T_n$ in its lower right part, a pattern of alternating $x$s and $w$s in its top right part, and a pattern of alternating $y$s and $z$s in its lower left part. This is best illustrated by the example of $M_{3,2}$ described in Figure \ref{F:M} below. Note that this is \cite[Figure 1]{Stock81}, where it is called $M_{12}$\ due to minor notational differences.  

\begin{figure}[hbbp]
  \centering
  \begin{tabular}{l l l l l l l l l l l l}
0 & 1 & 1 & 0 & 1 & 1 & 0 & 0 & $w$ & $x$ & $w$ & $x$ \\
0 & 0 & 1 & 1 & 0 & 1 & 1 & 0 & $x$ & $w$ & $x$ & $w$ \\
0 & 0 & 0 & 1 & 1 & 0 & 1 & 1 & $w$ & $x$ & $w$ & $x$ \\
1 & 0 & 0 & 0 & 1 & 1 & 0 & 1 & $x$ & $w$ & $x$ & $w$ \\
0 & 1 & 0 & 0 & 0 & 1 & 1 & 0 & $w$ & $x$ & $w$ & $x$ \\
0 & 0 & 1 & 0 & 0 & 0 & 1 & 1 & $x$ & $w$ & $x$ & $w$ \\
1 & 0 & 0 & 1 & 0 & 0 & 0 & 1 & $w$ & $x$ & $w$ & $x$ \\
1 & 1 & 0 & 0 & 1 & 0 & 0 & 0 & $x$ & $w$ & $x$ & $w$ \\
$y$ & $z$ & $y$ & $z$ & $y$ & $z$ & $y$ & $z$ & 0 & 1 & 1 & 0\\ 
$z$ & $y$ & $z$ & $y$ & $z$ & $y$ & $z$ & $y$ & 0 & 0 & 1 & 1\\
$y$ & $z$ & $y$ & $z$ & $y$ & $z$ & $y$ & $z$ & 0 & 0 & 0 & 1\\
$z$ & $y$ & $z$ & $y$ & $z$ & $y$ & $z$ & $y$ & 1 & 0 & 0 & 0
\end{tabular}
\caption{$M_{3,2}$}
\label{F:M} 
	\end{figure}
	
Given $0\leq n<m$, we obtain a digraph with vertices $\{v_1,\ldots,v_{2^m+2^n}\}$ when we set each of $w,x,y,z$ to either 0 or 1, and use $M_{m,n}$ as the adjacency matrix in the obvious way. Following \cite{Stock81}, we define six families of pairs of digraphs by specifying the values of $w,x,y,z$ as described in Figure \ref{F:families}. Note that the family of pairs from \cite{Stock77} arise here as $(D_{m,0},D^*_{m,0})$.

\begin{figure}[hbbp]
  \centering
  \begin{tabular}{| c | c | c | c | c |}
	\hline
Digraph & $w$ & $x$ & $y$ & $z$ \\ \hline 
$A_{m,n}$ & 1 & 0 & 0 & 0 \\ \hline
$A^*_{m,n}$ & 0 & 1 & 0 & 0 \\ \hline
$B_{m,n}$ & 0 & 0 & 1 & 0  \\ \hline 
$B^*_{m,n}$ & 0 & 0 & 0 & 1 \\\hline
$C_{m,n}$ & 1 & 0 & 1 & 0 \\\hline
$C^*_{m,n}$ & 0 & 1 & 0 & 1 \\\hline
$D_{m,n}$ & 1 & 0 & 0 & 1 \\\hline
$D^*_{m,n}$ & 0 & 1 & 1 & 0 \\\hline
$E_{m,n}$ & 1 & 1 & 1 & 0 \\\hline
$E^*_{m,n}$ & 1 & 1 & 0 & 1 \\\hline
$F_{m,n}$ & 1 & 0 & 1 & 1 \\\hline
$F^*_{m,n}$ & 0 & 1 & 1 & 1\\\hline 
\end{tabular}
\caption{Six families}
\label{F:families} 
	\end{figure}

\section{The tally-spectra of Stockmeyer graphs}\label{S:games}
It was shown in \cite{Stock81} that if $Z\in\{A,B,C,D,E,F\}$, and  $0\leq n< m$ then the pair $(Z_{m,n},Z^*_{m,n})$ constitutes a counterexample to the reconstruction conjecture for digraphs, in other words, they are not isomorphic, but they cannot be distinguished by comparing decks. We omit the proofs, as they can be found in the original article. We will show in this section that $(Z_{m,n},Z^*_{m,n})$ can always be distinguished in the 2-colour Seurat game $\G^2(Z_{m,n},Z^*_{m,n})$. In other words, $\forall$ always has a winning strategy. Though the exact reasons for this vary to some extent between the families, the arguments here come down to examining tally-spectra, and we will need some analysis of the tally-sequences of vertices in these digraphs. With that in mind, we proceed to a technical lemma and some corollaries. 

\begin{lem}\label{L:tally}
In the graph $T_n$, the tallies of the first $2^{n-1}$ vertices must all be $(2^{n-1}-1, 2^{n-1})$, and the tallies of the second $2^{n-1}$ vertices must all be $(2^{n-1}, 2^{n-1}-1)$.
\end{lem}
\begin{proof}
By \cite[Lemma 1(b)]{Stock77}, which we essentially generalize as Lemma \ref{L:halves} below, the out-degree of the first $2^{n-1}$ vertices is $2^{n-1}$, and the out-degree of the second $2^{n-1}$ vertices is $2^{n-1}-1$. Moreover, by elementary number theory, between every pair of distinct vertices in $T_n$ there must be exactly one edge. So, for each vertex $v$ of $T_n$, we must have $\ind(v)+\outd(v) = 2^n-1$, and the result follows immediately.
\end{proof}

\begin{cor}\label{C:tallyM}
In the graph with adjacency matrix $M_{m,n}$, the tallies of the vertices are described in the table in Figure \ref{F:tallyM}.
\end{cor}
\begin{proof}
This follows from Lemma \ref{L:tally}. The vertices with indices from $\{1,\ldots, 2^{m}\}$ form a copy of $T_m$, and so the lemma tells us the in- and out-degrees for these elements with respect to each other. The definition of $M_{m,n}$ tells us that each element gets an additional $2^{n-1}y + 2^{n-1}z$ to its in-degree, and an additional $2^{n-1}w + 2^{n-1}x$ to its out-degree. Similarly, the vertices with indices from $\{2^m + 1,\ldots, 2^m+2^n\}$ form a copy of $T_n$, which tells us the in- and out-degrees of these elements relative to each other, and each vertex gets an additional $2^{m-1}y + 2^{m-1}z$ to its in-degree, and $2^{m-1}w + 2^{m-1}x$ to its out-degree.
\end{proof}

\begin{figure}
\begin{tabular}{c | c }
 vertex index & tally \\
\hline
$\{1,\ldots ,2^{m-1}\}$ & $(2^{m-1} + 2^{n-1}y+2^{n-1}z - 1 \phantom{x},\phantom{x} 2^{m-1}+2^{n-1}w + 2^{n-1}x)$\\
$\{2^{m-1}+1,\ldots,2^m\}$ & $(2^{m-1} + 2^{n-1}y+2^{n-1}z \phantom{x},\phantom{x} 2^{m-1}+2^{n-1}w + 2^{n-1}x -1)$\\ 
$\{2^m + 1,\ldots, 2^m + 2^{n-1}\}$ & $(2^{n-1} + 2^{m-1}w+2^{m-1}x - 1 \phantom{x},\phantom{x} 2^{n-1}+2^{m-1}y + 2^{m-1}z)$\\
$\{2^m + 2^{n-1},\ldots, 2^m+2^n\}$ & $(2^{n-1} + 2^{m-1}w+2^{m-1}x \phantom{x},\phantom{x} 2^{n-1}+2^{m-1}y + 2^{m-1}z - 1)$
\end{tabular}
\caption{The tallies of the vertices in the graph with adjacency matrix $M_{m,n}$.}
\label{F:tallyM} 
\end{figure}

\begin{cor}\label{C:u-and-v}
Let $0\leq n < m$, and let $G$ be a graph with adjacency matrix $M_{m,n}$ such that $w,x,y,z\in\{0,1\}$ (e.g. one of the Stockmeyer graphs from Figure \ref{F:families}). For convenience, we assume the vertices of $G$ are the numbers $\{1,\ldots,2^m+2^n\}$. Let $u$ be a vertex from $\{1,\ldots, 2^m\}$, and let $v$ be a vertex from $\{2^m+1,\ldots, 2^m+ 2^n\}$. Then, if $n\geq 1$, $u$ and $v$ have the same tally if and only if both the following conditions hold:
\begin{enumerate}[(i)]
\item $u\in \{1,\ldots, 2^{m-1}\}$ and $v\in\{2^m+1,\ldots, 2^m+ 2^{n-1}\}$, or $u\in \{2^{m-1} + 1,\ldots, 2^m\}$ and $v\in \{2^m + 2^{n-1},\ldots, 2^m+2^n\}$.  
\item $y+z = 1$ and $w+x = 1$.
\end{enumerate}
If $n= 0$  the $T_n$ part of $G$ has a single node so $v=2^m+1$ is odd.  Then $u$ and $v$ have the same tally if and only if one of the following holds:
\begin{enumerate}[(i)]
\item $u\in \{1,\ldots, 2^{m-1}\}$ and $u + v$ is even (i.e. $u$ is odd) and $w=0$, $x=1$, $y = 1$ and $z = 0$.
\item $u\in \{1,\ldots, 2^{m-1}\}$ and $u + v$ is odd (i.e. $u$ is even) and $w=1$, $x=0$, $y=0$, and $z = 1$.
\item $u\in \{2^{m-1}+1,\ldots, 2^{m}\}$ and $u + v$ is even (i.e. $u$ is odd) and $w=1$, $x=0$, $y = 0$ and $z = 1$.
\item $u\in \{2^{m-1}+1,\ldots, 2^{m}\}$ and $u + v$ is odd (i.e. $u$ is even) and $w=0$, $x = 1$, $y=1$ and $z = 0$.
\end{enumerate}
\end{cor}
\begin{proof}
Suppose first that $n\geq 1$. If $u\in \{1,\ldots, 2^{m-1}\}$ then the in-degree of $u$ is odd, and the in-degree of $v$ is odd if and only if $v\in\{2^m + 1,\ldots, 2^m + 2^{n-1}\}$. Similarly, if $u\in \{2^{m-1}+1,\ldots,2^m\}$ then its in-degree is even, and $v$'s in-degree is even if and only if $v\in\{2^m + 2^{n-1},\ldots, 2^m+2^n\}$. Thus we have the necessity of (i). 

Suppose then that (i) holds, and let $u\in \{1,\ldots, 2^{m-1}\}$ and  $v\in\{2^m + 1,\ldots, 2^m + 2^{n-1}\}$. Then the in-degree of $u$ is $2^{m-1} + 2^{n-1}y+2^{n-1}z - 1$, and the in-degree of $v$ is $2^{n-1} + 2^{m-1}w+2^{m-1}x - 1$. Noting that $n<m$, we have
\begin{align*}
&2^{m-1} + 2^{n-1}y+2^{n-1}z - 1 = 2^{n-1} + 2^{m-1}w+2^{m-1}x - 1\\
\iff& 2^{m-n} + y + z = 1 + 2^{m-n}w + 2^{m-n}x\\
\iff& y+z - 1 = 2^{m-n}(w+x-1), 
\end{align*}  
and this occurs if and only if $y+z = 1$ and $w+x = 1$ (remember that $w,x,y,z$ can be only 0 or 1). Moreover, if this property holds then it's easy to check the out-degrees will also be the same. The argument for $u\in \{2^{m-1}+1,\ldots,2^m\}$ and  $v\in\{2^m + 2^{n-1},\ldots, 2^m+2^n\}$ is essentially the same.

Finally, if $n = 0$, then $v=2^m+1$ and exactly half of the elements of $\{1,\ldots,2^m\}$ have even sum with $v$. So the tally of $v$ is $(2^{m-1}w + 2^{m-1}x, \; 2^{m-1}y+2^{m-1}z)$.  

Suppose first that $u$ is odd. Then the tally of $u$ must be either $(2^{m-1}+y-1, \; 2^{m-1}+w)$, when $u\in  \{1,\ldots, 2^{m-1}\}$, or $(2^{m-1}+y, \; 2^{m-1}+w - 1)$, when $u\in \{2^{m-1}+1,\ldots, 2^{m}\}$. Suppose first that $\tau(u)= (2^{m-1}+y-1,\; 2^{m-1}+w)$. Then to have $\tau(u)=\tau(v)$ we must have $y=1$ and $w=0$. It then follows that we must have $x=1$ and $z=0$. Thus in this case $\tau(u)=\tau(v)$ if and only if $w=0$, $x=1$, $y = 1$ and $z = 0$, as claimed. Similar reasoning applies to the remaining case when $u$ is odd, and to the two cases where $u$ is even. 
\end{proof}

\begin{cor}\label{C:isom}
Let $Z\in\{C,D\}$, and let $G$ be either $Z_{m,n}$ or $Z^*_{m,n}$ for $0<n<m$. Let $u\in\{1,\ldots,2^m\}$, and let $v\in\{2^m+1,\ldots,2^m+2^n\}$. Let $S$ be the half of $\{1,\ldots,2^m\}$ containing $u$ (so $S = \{1,\ldots,2^{m-1}\}$ or $\{2^{m-1}+1,\ldots,2^m\}$), and define $S'$ to be the half of $\{2^m+1,\ldots,2^m+2^n\}$ containing $v$. Suppose $\tau(u)=\tau(v)$. Then the subgraph of $G$ generated by $S\cup S'$ is isomorphic to either $Z_{(m-1),(n-1)}$ (if $G = Z_{m,n}$) or $Z^*_{(m-1),(n-1)}$ (if $G = Z^*_{m,n}$). 
\end{cor}
\begin{proof}
Assume first that $S = \{1,\ldots,2^{m-1}\}$. By Corollary \ref{C:u-and-v}, we must have $S' = \{2^m+1,\ldots,2^m+2^{n-1}\}$. Now, $S$ and $S'$ are isomorphic to $T_{m-1}$ and $T_{n-1}$ respectively, and the pattern of edges between $S$ and $S'$ in $G$ is the same as for $T_{m-1}$ and $T_{n-1}$ in either $Z_{(m-1),(n-1)}$ or $Z^*_{(m-1),(n-1)}$, depending on what $G$ is. The argument for when $S= \{2^{m-1}+1,\ldots,2^m\}$ is essentially the same.
\end{proof}

Note that, by Corollary \ref{C:u-and-v}, the assumption that $\tau(u)=\tau(v)$ in the result above excludes the possibility that $Z\in\{A,B,E,F\}$. We are now in position to build up some facts about tally-spectra in our families of graphs.

\begin{lem}\label{L:tallyT}
For all $n\geq 0$, the tally-sequences of the elements of $T_n$ are distinct.
\end{lem}
\begin{proof}
This follows by iterating the argument from the proof of Lemma \ref{L:tally}, noting that by dividing the vertices of $T_n$ into two halves as in that proof we get two copies of $T_{n-1}$.
\end{proof}

Note that it follows immediately from the lemma above and Proposition \ref{P:tally-isom} that $T_n$ has only the identity automorphism. This is \cite[Lemma 1(c)]{Stock77}, and the argument given there amounts to showing the tally sequences of the vertices are all different.

\begin{lem}\label{L:no-corr}
Let $m\geq 1$ and let $0\leq n < m$. Then for all $Z\in\{A,B,E,F\}$, there is no $u$ in the $T_m$ part of $Z_{m,n}$ and $v$ in the $T_n$ part of $Z_{m,n}$ such that the tally of $u$ is the same as the tally of $v$. The same is also true for $Z^*_{m,n}$.  
\end{lem}
\begin{proof}
This follows immediately from Corollary \ref{C:u-and-v}. Specifically, note the values of $w,x,y,z$ from Figure \ref{F:families}.
\end{proof}

\begin{lem}\label{L:no-corr2}
Let $m\geq 1$ and let $0\leq n < m$. Then for all $Z\in\{A,B,C,E,F\}$, the tally-sequences of vertices of $Z_{m,n}$ are all distinct, and the same is true for $Z^*_{m,n}$.
\end{lem}
\begin{proof}
For $Z\in\{A,B,E,F\}$ this is straightforward: By Lemma \ref{L:no-corr}, no element of the $T_m$ part of $Z_{m,n}$ can have the same tally-sequence as an element from the $T_n$ part of $Z_{m,n}$, and the same goes for $Z^*_{m,n}$. Moreover, using Corollary \ref{C:tallyM}, it's easy to see that the first half of the $T_m$ part of $Z_{m,n}$ all have the same tally, $t_1$ say, as do the second half, $t_2$ say, and $t_1\neq t_2$. The same applies to the $T_n$ part of $Z_{m,n}$. Now, by Lemma \ref{L:tallyT}, the tally-sequences of the elements of the $T_m$ parts and the $T_n$ parts must all be distinct from each other, and so every element of $Z_{m,n}$ has a unique tally-sequence. The same argument applies to $Z^*_{m,n}$.

For $Z = C$ we need to use Corollary \ref{C:u-and-v}. If $n = 0$ then the $T_n$ part of $C_{m,n}$ contains only a single vertex, $v$ say, and the corollary tells us that $v$ cannot have the same tally as any vertex of the $T_m$ part of $C_{m,n}$. The claim then follows by the same argument we used in the first part. 

Suppose then that $n > 0$, let $u$ be a vertex of the $T_m$ part of $C_{m,n}$, and let $v$ be a vertex of the $T_n$ part of $C_{m,n}$.  Suppose for a contradiction that $\vec{\tau}(u)=\vec{\tau}(v)$. Define $S_0$ to be $\{1,\ldots, 2^m\}$, which contains $u$, and define $S'_0$ to be $\{2^m+1,\ldots, 2^m+2^n\}$, which contains $v$. Define $S_1$ to be the half of $S_0$ containing $u$, and define $S'_1$ to be the half of $S'_0$ containing $v$. Now, assuming $n- 1 \geq 1$, we can define $S_2$ and $S'_2$ to be, respectively, the halves of $S_1$ and $S'_1$ containing $u$ and $v$. Provided $n-2\geq 1$ we can define $S_3$ and $S'_3$ similarly. In general, assuming we have defined $S_i$ and $S'_i$ and that $n-i\geq 1$, we define $S_{i+1}$ and $S'_{i+1}$ to be, respectively, the halves of $S_i$ and $S'_i$ containing $u$ and $v$. The crucial observation, which we will shortly prove, is that the tally-sequences of every element in $S_{i+1}$ and $S'_{i+1}$ agree for the first $(i+1)$ steps (i.e. $\vec{\tau}^i(w) = \vec{\tau}^i(w')$ for all $w\in S_{i+1}$ and $w'\in S'_{i+1}$).

To see why this is true, we use induction on $i$, starting with $i = 0$. The base case follows immediately from Corollary \ref{C:tallyM}. For the inductive step, suppose the claim is true for $i= k$, and also that $S_{k+1}$ and $S'_{k+1}$ are both defined (i.e. that $n-k\geq 1$). By Corollary \ref{C:isom}, the subgraph generated by $S_k\cup S'_k$ is isomorphic to $C_{(m-k),(n-k)}$. Moreover, if $u'\in S_{k+1}$ and $v'\in S'_{k+1}$, then the $(k+1)$th term of their tally sequences will be their tallies in this subgraph. By Corollaries \ref{C:tallyM} and \ref{C:u-and-v} these are the same, which gives the result.

Suppose then that we have constructed $S_n$ and $S'_n$. At this point $S'_n$ contains only a single element, and $S_n$ is a copy of $T_{(m-n)}$. Let $G$ be the graph generated by $S_n\cup S'_n$. By the argument that proves the $n = 0$ part of Corollary \ref{C:u-and-v}, the tallies of $u$ and $v$ relative to $G$ cannot be the same, and thus $\vec{\tau}(u)\neq\vec{\tau}(v)$ after all. 

Since we have now proved that no element of the $T_m$ part of $C_{m,n}$ can have the same tally-sequence as an element from the $T_n$ part, the fact that every element has a distinct tally-sequence now follows by the argument used for $Z\in\{A,B,E,F\}$ earlier. The argument for $C^*_{m,n}$ is similar. 
 
\end{proof}

Lemma \ref{L:no-corr2} seems to only be occasionally true for $D^*_{m,n}$ and $D_{m,n}$, and at different times for $D$ and $D^*$.  To get round this problem we have do a bit more work.

\begin{lem}\label{L:tallyD}
Let $m\geq 1$ and let $0\leq n < m$. Let $i\in\{1,\ldots,2^m+2^n\}$, and let $v$ and $w$ be the corresponding vertices in $D_{m,n}$ and $D^*_{m,n}$ respectively (recall that we define the pairs of Stockmeyer graphs using identically labeled sets of vertices, and this defines a correspondence). Then the tally-sequences of $v$ and $w$ agree in their first $n$ places. In other words, $\vec{\tau}^{n-1}(v) = \vec{\tau}^{n-1}(w)$.
\end{lem}
\begin{proof}
This is proved by applying Corollaries \ref{C:tallyM} and \ref{C:isom} repeatedly, using the fact that in both $D_{m,n}$ and $D^*_{m,n}$ exactly one in each pair $(w,x)$ and $(y,z)$ is 1. At stage $k$ we are effectively working with $D_{(m-k),(n-k)}$ and $D^*_{(m-k),(n-k)}$, and the logic holds up till $k = n$.  
\end{proof}

We will need the following purely number theoretic result. It is essentially a corollary of \cite[Lemma 1(b)]{Stock77}.

\begin{lem}\label{L:halves}
Let $1\leq x,n < \omega$. Let $X=(x_1,\ldots, x_{2^n})$ be defined by $x_i = x + 2(i-1)$ for all $i\in\{1,\ldots, 2^n\}$. Then elements from $(x_1,\ldots, x_{2^{n-1}})$ dominate $2^{n-1}$ elements from $X$, and are dominated by $2^{n-1}-1$ elements from $X$, and elements from $(x_{2^{n-1}+1},\ldots,x_{2^n})$ dominate $2^{n-1}-1$ elements from $X$, and are dominated by $2^{n-1}$ elements from $X$.
\end{lem}
\begin{proof}
First, let $i,j\in \{1,\ldots,2^n\}$. Then 
\begin{align*} \odd(x_j-x_i)&=\odd(x+2(j-1) - (x+2(i-1)))\\&
=\odd(j-i).\end{align*} 
Thus for $i,j\in \{1,\ldots,2^n\}$, domination between $x_i$ and $x_j$ is the same as domination between $i$ and $j$. So it suffices to prove that elements from $\{1,\ldots, 2^{n-1} \}$ and $\{2^{n-1}+1,\ldots, 2^n\}$ dominate and are dominated by, respectively, $2^{n-1}$ and $2^{n-1}-1$, and $2^{n-1} - 1$ and $2^{n-1}$ elements from $\{1,\ldots,2^n\}$. This is essentially what \cite[Lemma 1(b)]{Stock77} says, but for convenience we provide the details.

Let $j' = 2i-j \mod 2^n$. Then $j'$ is either $2i-j$ or $2i-j \pm 2^n$. So $\odd(j'-i)$ is either $\odd(i-j)$, or $\odd(i-j \pm2^n)$, which, by Lemma \ref{L:num} below, is also $\odd(i-j)$ when working modulo 4, unless $i- j\in \{0,\pm 2^{n-1}\}$.

It follows that for fixed $i\in\{0,\ldots,2^n - 1\}$, for each $j\in \{1,\ldots,2^n \}\setminus \{i,i\pm 2^{n-1}\}$ we either have $i$ dominating $j$, and being dominated by $j'$, or vice versa. Thus from the pairs $(j,j')$ where $i-j\notin\{0,\pm 2^{n-1}\}$ we see that $i$ dominates and is dominated by $2^{n-1}-1$ elements of  $\{1,\ldots,2^n \}\setminus \{i,i\pm 2^{n-1}\}$. 

Finally, if $i\leq 2^{n-1}$ then $i+ 2^{n-1}\in \{1,\ldots,2^n\}$, and $\odd(i+2^{n-1} - i) = \odd(2^{n-1}) = 1 \equiv_4 1$, and so $i$ dominates $i+ 2^{n-1}$ (which is not included in the previous count). On the other hand, if $i>2^{n-1}$, then $i- 2^{n-1}\in \{1,\ldots,2^n\}$, and $\odd(i-2^{n-1} - i) = \odd(-2^{n-1}) = -1 \equiv_4 3$, and so $i$ is dominated by $i- 2^{n-1}$ (which is again not previously counted). Thus, when $i\leq 2^{n-1}$ it dominates an additional number, and when $i> 2^{n-1}$ it is dominated by an additional number. This gives the result.
\end{proof}

\begin{lem}\label{L:num}
For all $i,j\in\{0,\ldots,2^n -1\}$, if $j- i \notin\{0,\pm 2^{n-1}\}$, then 
\[\odd((j-i) \pm 2^n) \equiv_4 \odd(j-i).\]
\end{lem}
\begin{proof}
Suppose $j-i = 2^kq$ for some $k< \omega$ and some odd $q$. Then, as $j-i$ is neither $2^{n-1}$ nor $-2^{n-1}$, we must have $k\leq n-2$, because $|j-i|$ is bounded by the possible choices of $i$ and $j$. So
\begin{align*}
\odd (j-i \pm 2^n) &= \odd(2^kq \pm 2^n)\\
&= \odd(2^k(q\pm 2^{n-k})) \\
&= \odd(q\pm 2^{n-k}) \\
&= q \pm 2^{n-k} \\
&\equiv_4 q,
\end{align*}
with the final modular equality holding because $n-k \geq 2$.
\end{proof}

\begin{lem}\label{L:Ddiff}
Let $m\geq 1$ and let $0\leq n < m$. Then the tally-spectra of $D_{m,n}$ and $D^*_{m,n}$ are not the same.
\end{lem}
\begin{proof}
Suppose first that $n = 0$, and let $v$ and $w$ be, respectively, the lone vertices in the $T_n$ parts of $D_{m,n}$ and $D^*_{m,n}$. Then, as noted in the penultimate paragraph of the proof of Corollary \ref{C:u-and-v}, the tallies of both $v$ and $w$ are $(2^{m-1}, 2^{m-1})$. Corollary \ref{C:u-and-v} also tells us that the vertices of $D_{m,0}$ with that tally are even numbers in $\{1,\ldots,2^{m-1}\}$, and the odd numbers in $\{2^{m-1}+1,\ldots, 2^m\}$, and that the vertices of $D^*_{m,n}$ with that tally are the odd numbers in $\{1,\ldots,2^{m-1}\}$, and the even numbers in $\{2^{m-1}+1,\ldots, 2^m\}$.  

We consider $D_{m,0}$ first. Define $X_0$ to be the set of even numbers from $\{1,\ldots,2^{m-1}\}$, define $Y_0$ to be the set of odd numbers from $\{2^{m-1}+1,\ldots, 2^m\}$, and define $G_0$ to be the subgraph induced by $X_0\cup Y_0\cup\{v\}$. If $m =1$ then $X_0$ and $Y_0$ are both empty, so we will assume that $m\geq 2$. Let $x\in X_0$. Then $x$ dominates exactly $2^{m-2}$ numbers from $Y_0$, and is dominated by the other $2^{m-2}$ (by the definition of domination). Moreover, by Lemma \ref{L:halves}, if $x\in  X_0\cap\{1,\ldots,2^{m-2}\}$ then $x$ dominates $2^{m-2}$ numbers from $X_0$, and is dominated by $2^{m-2}-1$, and if $x\in X_0\cap \{2^{m-2}+1,\ldots,2^{m-1}\}$ then $x$ dominates $2^{m-2}-1$ numbers from $X_0$, and is dominated by $2^{m-2}$. Similar arguments apply to $y\in Y_0$. Moreover, in $G_0$ there is an edge from $v$ to every vertex in $X_0$, and an edge from every vertex of $Y_0$ to $v$. Putting all this together, a little calculation reveals that the tallies of the vertices of $G_0$ relative to $G_0$ are as follows:
\[\begin{tabular}{c | c }
 vertex & tally \\
\hline
$X_0\cap\{1,\ldots ,2^{m-2}\}$ & $(2^{m-1}, 2^{m-1})$\\
$X_0\cap\{2^{m-2}+1,\ldots,2^{m-1}\}$ & $(2^{m-1} + 1, 2^{m-1}-1)$\\ 
$Y_0\cap\{2^{m-1} + 1,\ldots, 2^{m-1}+2^{m-2}\}$ & $(2^{m-1}-1, 2^{m-1}+1)$\\
$Y_0\cap\{2^{m-1}+2^{m-2} +1,\ldots, 2^m\}$ & $(2^{m-1} , 2^{m-1})$\\
$v$ & $(2^{m-1} , 2^{m-1})$
\end{tabular}\] 
We can now define \[X_1 = X_0\cap\{1,\ldots ,2^{m-2}\},\]
\[Y_1 = Y_0\cap \{2^{m-1}+2^{m-2} +1,\ldots, 2^m\}=Y_0\cap\{2^m-2^{m-2}+1, 2^m\},\] and 
\[G_1 = X_1\cup Y_1 \cup\{v\}.\]
Using the same logic as before, we see that, relative to $G_1$, the vertices in $X_1\cap\{1,\ldots ,2^{m-3}\}$ and $Y_1\cap \{2^m- 2^{m-3} +1,\ldots, 2^m\}$ are the ones whose tallies agree with that of $v$. In general, we define $X_k$ to be the even members of $\{1,\ldots,2^{m-k-1}\}$, and we define $Y_k$ to be the odd members of $\{2^m - 2^{m-(k+1)}+1,\ldots,2^m \}$. We can continue in this way till we reach $G_{m-2} = X_{m-2}\cup Y_{m-2}\cup\{v\}$, which contains precisely those elements $u$ of $D_{m,0}$ such that $\vec{\tau}^{m-2}(u) = \vec{\tau}^{m-2}(v)$. At this point $X_{m-2}$ is just $\{2\}$, and $Y_{m-2}$ is just $\{2^m-1\}$. Now, $v$ dominates $2$, and is dominated by $2^m-1$, and $2$ dominates $2^m-1$, so each vertex of $G_{m-2}$ has tally $(1,1)$ relative to $G_{m-2}$. Thus their tally-sequences start repeating here, and so are equal.

Taking stock, we have proved that if $m=1$, then $v$ is the only vertex of $D_{m,0}$ with its tally-sequence, and if $m\geq 2$ then the vertices with the same tally-sequence as $v$ are precisely $\{v, 2, 2^m-1\}$. We can now run a similar argument on $D^*_{m,0}$ and $w$. As before, if $m = 1$ then $w$ is the only element with its tally-sequence, so, assuming $m\geq 2$, we define $X^*_0$ and $Y^*_0$ to be, respectively, the elements of $\{1,\ldots,2^{m-1}\}$ and $\{2^{m-1}+1,\ldots,2^m\}$ with the same tally as $w$. In this case $X^*_0$ turns out to contain precisely the \emph{odd} numbers, and $Y^*_0$ precisely the \emph{even} numbers. Aside from the parity flip, the argument can now be run in the same way as before, till we obtain 
\[G^*_{m-2} = X^*_{m-2}\cup Y^*_{m-2}\cup \{w\} = \{1, 2^m, w\}\] 
as the set of vertices of $D^*_{m,0}$ whose tally-sequences agree with that of $w$ in their first $m-1$ places. But now there is a change, because $w$ dominates $1$ and is dominated by $2^m$, but $2^m$ dominates $1$, so $w$ has tally $(1,1)$ relative to $G^*_{m-2}$, but $2^m$ dominates 1, and so $1$ and $2^m$ have tallies $(2,0)$ and $(0,2)$ respectively. So $w$ is the only vertex of $D^*_{m,0}$ with its tally-sequence, and this starts repeating when it gets to $(0,0)$, which it does immediately after $(1,1)$.

Now, it's easy to see that $\vec{\tau}^{m-2}(v) = \vec{\tau}^{m-2}(w)$, and it follows from the discussion above that these tally-sequences disagree after this point. Moreover, we showed that every vertex of $D_{m,n}$ whose tally-sequence had agreed with that of $w$ up to this point (the vertices $\{v,2, 2^m-1\}$) has the same tally-sequence as $v$. Thus the tally-spectra of $D_{m,0}$ and $D^*_{m,0}$ must be different when $m\geq 2$, as there is no vertex of $D_{m,n}$ with the same tally-sequence as $w$. Finally, a direct check reveals the same is true when $m=1$.

Now, to continue, suppose $n>0$. We will reduce this to the $n=0$ case. It follows from Corollary \ref{C:u-and-v} that given $v$ in the $T_n$ part of $D_{m,n}$, the graph generated by the set of vertices $u$ of $D_{m,n}$ such that $u\neq v$ and $\vec{\tau}^{n-1}(u)=\vec{\tau}^{n-1}(v)$ is isomorphic to $T_{m-n}$. Consequently, depending on the parity of $v$, the graph $G$ of vertices of $D_{m,n}$ whose tally-sequences agree with that of $v$ in their first $n$ places will either be isomorphic to $D_{m-n,0}$ (when $v$ is odd), or $D^*_{m-n,0}$ (when $v$ is even). Moreover, this isomorphism will be an order isomorphism on the poset induced on the graphs by thinking about the sizes of numbers. The same is also true for $w$, where $w$ is the correspondent of $v$ in the $T_n$ part of $D^*_{m,n}$, giving us a graph $G^*$ isomorphic as a graph and order isomorphic to either $D_{m-n,0}$, or $D^*_{m-n,0}$. The only difference is that $G\cong D_{m-n,0} \iff G^*\cong D^*_{m-n,0}$. 

So, to find a vertex of $D^*_{m,n}$ with the same tally-sequence as $v$ we must find a vertex of $G^*$ with the same tally-sequence as $v$ (considered as a vertex of $G$). But we know from the $n=0$ case that there is no such vertex.           
\end{proof}

\begin{thm}\label{T:ZZ}
Let $0\leq n < m$. Then for all $Z\in\{A,B,C,D,E,F\}$, $\forall$ has a winning strategy in $\G^2(Z_{m,n}, Z_{m,n}^*)$.
\end{thm}
\begin{proof}
By Lemma \ref{L:Ddiff}, the tally-spectra of $D_{m,n}$ and $D^*_{m,n}$ are not the same, so the result for $Z=D$ follows from Corollary \ref{C:distinct}. Suppose then that $Z\in\{A,B,C,E,F\}$.

By Lemma \ref{L:2-colours}, if $\forall$ colours an element of $Z_{m,n}$ red, then $\exists$ must respond by colouring an element of $Z^*_{m,n}$ red, and the two elements must have the same tally-sequences. Note that, as $Z_{m,n}\not\cong Z^*_{m,n}$, matching elements by tally-sequence cannot be an isomorphism. 

Define the map $h:Z_{m,n}\to Z^*_{m,n}$ by sending vertices of $Z_{m,n}$  to the unique vertex of $Z^*_{m,n}$ with the same tally-sequence. If $h$ is not well defined then $\forall$ has a strategy by Corollary \ref{C:distinct}, so assume $h$ can be defined like this. By Lemma \ref{L:no-corr2} $h$ is a bijection, but it cannot be an isomorphism as $Z_{m,n}\not\cong Z^*_{m,n}$.  So there must be a pair of vertices $u,v$ of $Z_{m,n}$ such that $h$ restricted to the subgraph generated by $\{u,v\}$ is not an isomorphism onto the subgraph generated by $\{h(u), h(v)\}$. It follows that $\forall$ can win by colouring $u$ red, then colouring $v$ blue, because $\exists$ must respond by colouring $h(u)$ red and $h(v)$ blue to match tally-sequences, but then there will be a disagreement about edges between colours.          
\end{proof}

\section{The reconstruction conjectures}\label{S:rec}

Consider first the degree-associated reconstruction conjecture, which for our purpose is most conveniently stated in the following form.

\begin{defn}\label{D:deg_rec}
The \textbf{degree-associated reconstruction conjecture for digraphs} is that if $G$ and $H$ are non-isomorphic digraphs, and if at least one of them has at least three vertices, then there is a pair $(x,y)\in \omega\times\omega$, and a digraph $F$ such that if 
\[S = \{u\in G: \tau(u) = (x,y)\text{ and }G\setminus\{u\} \cong F\}\] and 
\[T = \{v\in H: \tau(v) = (x,y)\text{ and }H\setminus\{v\} \cong F\},\] 
then $|S|\neq|T|$. 
\end{defn}

The data of the deck of a digraph along with the in-degree, out-degree pair for each of the deleted vertices is often known as its \textbf{degree-associated deck}. 

\begin{thm}\label{T:stronger}
If the degree-associated reconstruction conjecture for digraphs is true, then whenever $G$ and $H$ are digraphs with $G\not\cong H$, there is a winning strategy for $\forall$ in $\G^3(G,H)$.
\end{thm}
\begin{proof}
Let $G$ and $H$ be digraphs with $G\not\cong H$. By Proposition \ref{P:1-colour}(\ref{One2}) and Proposition \ref{P:constraints}(S1),  we can assume that $|G|=|H|\geq 3$.  So, assuming the degree-associated reconstruction conjecture is true, we can choose a pair $(x_0,y_0)\in \omega\times\omega$ and a digraph $F_0$ satisfying the conditions from Definition \ref{D:deg_rec}. Let $S_0$ and $T_0$ be as in that definition, so 
\[S_0 = \{u\in G: \tau(u) = (x_0,y_0)\text{ and }G\setminus\{u\} \cong F_0\}\] and 
\[T_0 = \{v\in H: \tau(v) = (x_0,y_0)\text{ and }H\setminus\{v\} \cong F_0\},\]   and suppose without loss of generality that $|S_0|>|T_0|$.  Consider the following strategy for $\forall$. First he colours $S_0$ red. Now $\exists$ must respond by colouring some subset $Y_0$ of $H$ red, and the tallies of vertices in this set must all be $(x_0,y_0)$, otherwise $\forall$ can force a win, by Corollary~\ref{C:2-colours}. By assumption, there must be some $u_0\in Y_0$ with $H\setminus \{u_0\}\not\cong F_0$. 

For his next move, $\forall$ then colours $H\setminus \{u_0\}$ blue. Then $\exists$ must respond by choosing $v_0\in S_0$ and colouring $G\setminus\{v_0\}$ blue. Define $G_1 = G\setminus\{v_0\}$, and define $H_1 = H\setminus \{u_0\}$, note that these are proper subgraphs. Then, by assumption, we have $G_1\cong F_0\not\cong H_1$, and so, again by assumption of the degree-associated reconstruction conjecture, we have a pair $(x_1,y_1)\in \omega\times\omega$ and a digraph $F_1$, such that, if 
\[S_1 = \{u\in G_1: \tau_{G_1}(u) = (x_1,y_1)\text{ and }G_1\setminus\{u\} \cong F_1\}\] and 
\[T_1 = \{v\in H_1: \tau_{H_1}(v) = (x_1,y_1)\text{ and }H_1\setminus\{v\} \cong F_1\},\] 
then $|S_1|\neq |T_1|$.

Note that both $G_1$ and $H_1$ are coloured blue, so $\forall$ can repeat his play as described above, \emph{mutatis mutandis}, with the other two colours, this time restricting himself to $G_1$ and $H_1$. Since $\exists$ must restrict her responses to $G_1$ and $H_1$ too, this produces $G_2\subset G_1$ and $H_2\subset H_1$, both coloured green say, with $G_2\not\cong H_2$. Repeating this play with $G_2$ and $H_2$, and again with $G_3\subset G_2$ and $H_3\subset H_2$, and so on, $\forall$ will, unless he wins before this point, eventually force a situation where there is $G_k\subset G$ and $H_k\subset H$ both coloured blue say, with $G_k\not\cong H_k$, and with $|G_k|=|H_k|\leq 4 = 2^2$. At this point he can force a win with the other two colours, as noted in Lemma \ref{L:rough_bound}.
\end{proof}

In Theorem~\ref{T:ZZ} we proved that a Stockmeyer pair $(Z, Z^*)$ is not a counterexample to Conjecture~\ref{conj}, because $\forall$ has a winning strategy in ${\bf G}^2(Z, Z^*)$, but the proof took some work.  It might be tempting to try to use a strategy similar to that used in the proof of Theorem~\ref{T:stronger} to obtain the same result, or at least the similar result for $\G^3$,  more easily by exploiting the fact that the degree-associated decks of $Z$ and $Z^*$ are known to be different.  However, such a proof would not work without the assumption of the degree-associated reconstruction conjecture.  Using the fact that the degree-associated decks are different, we could obtain $G_1\cong F_0\not\cong H_1$, but we would need the degree-associated reconstruction conjecture to find a suitable $F_1$ and continue.  In other words, it is not enough that the degree-associated decks of $(Z, Z^*)$ differ, we would also need to know that the degree-associated decks of various non-isomorphic pairs of subgraphs also differed.

Theorem \ref{T:stronger} has an analogue for the version of Conjecture \ref{conj} obtained by replacing digraphs with undirected graphs. 

\begin{thm}\label{T:stronger2}
If the reconstruction conjecture is true, then whenever $G$ and $H$ are graphs with $G\not\cong H$, there is a winning strategy for $\forall$ in $\G^3(G,H)$.
\end{thm}
\begin{proof}
This is essentially the same as the proof of Theorem \ref{T:stronger}.
\end{proof}

Using the connection between $k$-colour Seurat games and monadic second-order logic with $k$ second-order and $2$ first-order variables, we obtain the following easy corollary.

\begin{cor}
If the reconstruction conjecture is true, then given graphs $G$ and $H$ with $G\not\cong H$, there is a sentence $\phi$ of  monadic second-order logic with up to $3$ second-order and up to $2$ first-order variables such that $G\models \phi$ and $H\not\models\phi$.
\end{cor}
\begin{proof}
By Proposition \ref{P:MSO}, if $\forall$ has a winning strategy in $\G^3(G,H)$ then he also has one in $\M^3_2(G,H)$. The result follows immediately, as Theorem \ref{T:stronger2} states that if the reconstruction conjecture is true, then $\forall$ has a winning strategy in $\G^3(G,H)$ whenever $G\not\cong H$.
\end{proof}

It is known that no fixed finite number $k$ of variables  is sufficient to distinguish all non-isomorphic graphs in the first-order counting logic $\mathsf{C}^k$ \cite{CFI92}, but we do not know of any similar result for monadic second-order logic, though this logic has been studied extensively in the context of graph theory (see e.g. \cite{CouEng12}). If such a result existed it would, by the above corollary, disprove the reconstruction conjecture.

We do not know if converses hold for either Theorem \ref{T:stronger} or Theorem \ref{T:stronger2}. In other words, if either Conjecture \ref{conj} or its undirected analogue are equivalent to the degree-associated reconstruction conjecture or the reconstruction conjecture, respectively. Suppose for the sake of argument that we want to prove the converse to Theorem \ref{T:stronger2}. We might reason as follows. We are assuming that $\forall$ has a winning strategy in  $\G^3(G,H)$ whenever $G\not\cong H$, and we want to prove that the reconstruction conjecture follows from this. So, in other words, we want to prove that if $G\not\cong H$, then the decks of $G$ and $H$ are different. By our assumption we can suppose that $\forall$ has a winning strategy in  $\G^3(G,H)$, so if we could prove that $\forall$ having such a strategy implies the decks must be different then we would have have our proof. However, this seems to be easier said than done. Indeed, we know it is false in the case of digraphs, because we have seen that the Stockmeyer families produce pairs of graphs where $\forall$ has a winning strategy in the 2-colour game (hence he also has a winning strategy in the 3-colour game), but which nevertheless have the same decks (though different degree-associated decks).

The exact relationship between the strength of looking at decks or degree-associated decks and the existence of a winning strategy for $\forall$ in some $k$-colour game as a means of distinguishing non-isomorphic graphs is also unclear. We saw in Section \ref{S:games}  that there are graphs with the same deck that can be distinguished in the  2-colour game,  but beyond this we are currently in the dark.


\end{document}